\documentclass[11pt,a4paper]{amsart}

\usepackage{amsmath,amsthm}
\usepackage{amssymb,amscd}
\usepackage{hyperref,verbatim}

\def\O{{\mathcal O}}

\def\Ind{{\rm Ind}}

\def\Res{{\rm Res}}

\theoremstyle{plain}
\newtheorem{lem}{Lemma}[section]
\newtheorem{prop}[lem]{Proposition}
\newtheorem{cor}[lem]{Corollary}
\newtheorem{thm}[lem]{Theorem}

\theoremstyle{definition}
\newtheorem{Def}[lem]{Definition}
\theoremstyle{remark}

\newtheorem*{eg*}{Example}

\newtheorem*{rem}{Remark}

\newcommand{\del}{\partial}
\newcommand{\dom}{\vartriangleright}
\newcommand{\domeq}{\trianglerighteq}
\newcommand{\up}{\text{$\uparrow$}}
\newcommand{\down}{\text{$\downarrow$}}
\newcommand{\soc}{\operatorname{soc}}
\newcommand{\wt}{\widetilde}
\newcommand{\rad}{\operatorname{rad}}
\newcommand{\Ext}{\operatorname{Ext}}

\newcommand{\Jan}{\operatorname{Jan}}
\newcommand{\q}{\mathfrak{q}}
\begin{document}

\title[Weight $2$ blocks]{Weight $2$ blocks of general linear groups \\ and modular Alvis-Curtis duality}

\author{Sibylle Schroll}
\address[S. Schroll]{Mathematical Institute, 24--29 St Giles', Oxford, OX1 3LB, United Kingdom.}
\email{schroll@maths.ox.ac.uk}

\author{Kai Meng Tan}
\address[K. M. Tan]{Department of Mathematics, National University of Singapore, 2, Science Drive 2, Singapore 117543.}
\email{tankm@nus.edu.sg}

\subjclass[2000]{20C33, 20C08, 17B37}

\thanks{The authors thank the Department of Mathematics, NUS, and Karin Erdmann and the Mathematical Institute, Oxford, for their hospitality during their respective visits in 2006.  They acknowledge support through a Marie Curie fellowship and by Academic Research Fund R-146-000-089-112 of NUS respectively.}
\date{October 2007}

\begin{abstract}
We obtain the structure of weight $2$ blocks and $[2:1]$-pairs of $\q$-Schur algebras, and compute explicitly the modular Alvis-Curtis duality for weight $2$ blocks of finite general linear groups in non-defining characteristic.
\end{abstract}
\maketitle

\section{Introduction}

Let ${\mathbb F}_q$ be a finite field with $q$ elements, for some prime power $q$, and let $G(q)$ be a
finite reductive group defined over ${\mathbb F}_q$. Alvis-Curtis duality was originally defined (see, for example,
\cite[Chapter 5]{DM})
as a character duality in the Grothendieck group of $G(q)$.
Recently there has been renewed interested in Alvis-Curtis duality due to
some remarkable new results. Cabanes and Rickard~\cite{CR} showed that Alvis-Curtis
duality is induced by a derived equivalence proving thus a conjecture of Brou\'e~\cite{B}. In their paper they
conjectured that this derived equivalence should in fact be a homotopy equivalence, and a proof of this has recently been announced by Okuyama~\cite{O}.

On the other hand, Alvis-Curtis duality for general linear groups in non-defining characteristic has been linked to its decomposition numbers~\cite{AS}. More precisely, let ${\rm GL}_n(q)$ be the finite
general linear group with coefficients in ${\mathbb F}_q$, and let $k$ be an algebraically closed field of positive
characteristic $\ell$, where $\ell$ does not divide $q$; then Alvis-Curtis duality is a duality operation
in the Grothendieck group of $k {\rm GL}_n(q)$. Unlike the characteristic 0 case where the Alvis-Curtis dual of
an element of the Grothendieck group is explicitly known, all one can generally say in characteristic $\ell$ is that
the Alvis-Curtis dual of an element of the Grothendieck group of $k {\rm GL}_n(q)$ is given by a linear combination
of the basis elements. The difficulty in the computation of the coefficients in this linear combination
follows from the fact that the knowledge of these coefficients is equivalent to the knowledge of the decomposition
numbers of general linear groups~\cite{AS}.
As the complete determination of the latter is a longstanding open problem in
modular representation theory, it is not surprising that in general
not much can be said about modular Alvis-Curtis duality.

It is natural to attempt to compute these coefficients for unipotent blocks of finite general linear groups which are fairly well understood.  The representation theory of these blocks is very closely related to that of $q$-Schur algebras.  Among the simplest and yet non-trivial blocks are those with weight $2$.  However, while weight $2$ blocks have been studied extensively in the context of symmetric groups, Iwahori-Hecke algebras and Schur algebras (see, for example, \cite{S,R,CT}), the same cannnot be said for the $q$-Schur algebras (even though experts believe the results for Schur algebras \cite{CT} should generalise).

In this paper, we first study the structures of weight $2$ blocks of $\q$-Schur algebras for a general root of unity $\q$.  For such blocks where the characteristic of the underlying field is not $2$, we obtain closed formulas for the decomposition numbers (as well as the corresponding $v$-decomposition numbers arising from the canonical basis of the Fock space representation of $U_v(\widehat{\mathfrak{sl}}_e)$), and combinatorial descriptions of their $\Ext$-quivers and the Jantzen filtrations and radical filtrations of their Weyl modules.  We also show that the composition length of a Weyl module in such a block is bounded above by $5$.

With the knowledge of the decomposition numbers, it is in principle possible to compute the integers arising from Alvis-Curtis duality, but we find this hard in practice, owing to the difficulty in obtaining nice closed formulas for the entries of the inverse of the decomposition matrix.  We thus determine these integers in a more roundabout way.  We prove that over a $[2:k]$-pair, almost all of these integers `remain unchanged', and for those that do change, we provide a description on how they change.  With this, we are able to describe these integers for an arbitrary weight $2$ block from our knowledge for the weight $2$ Rouquier blocks (which can easily be computed using the closed formulas for $v$-decomposition numbers and their `inverses' obtained by Leclerc and Miyachi \cite{LM}).

This paper is organised as follows:  in the next section, we provide the relevant background and introduce the notations used in this paper.  We also prove some useful results relating to the $v$-decomposition numbers arising from the canonical basis of the Fock space representation of $U_v(\widehat{\mathfrak{sl}}_e)$, which may be of independent interests to the readers; results particularly worth mentioning are Corollary \ref{C:full}, Theorem \ref{T:R-H} and Proposition \ref{P:degree}.  In Sections $3$ and $4$, we study the structure of weight $2$ blocks and $[2:1]$-pairs respectively, while in section $5$, we compute the integers arising from Alvis-Curtis duality for the weight $2$ blocks.

\section{Background and Notations}

In this section we provide the relevant background and introduce the notations to be used throughout the paper.

\subsection{Partitions}

Let $\Lambda(n)$ be the set of all compositions of $n$ and ${\mathcal P}_n$ the subset of $\Lambda(n)$ consisting
of all partitions of $n$. Here, a composition $\lambda =(\lambda_1, \lambda_2, \dotsc, \lambda_k)$ of $n$ is
a finite sequence of non-negative integers such that $\sum_{i=1}^k \lambda_i = n$, and $\lambda$ is a
partition of $n$ if in addition, $\lambda_1 \geq \lambda_2 \geq \dotsb \geq \lambda_k$.
Denote by $l(\lambda)$, the number of non-zero parts of $\lambda \in \Lambda(n)$.

A sequence of $\beta$-numbers for a partition $\lambda$ is a strictly decreasing sequence of
 non-negative integers
$\beta = (\beta_1, \beta_2, \ldots, \beta_s)$ where $s \geq l(\lambda)$ and
$$
\beta_i =
\begin{cases}
\lambda_i + s -i, &\text{if } 1 \leq i \leq l(\lambda); \\
s-i, &\text{if } s > l(\lambda).
\end{cases}
$$

Let $e$ be a positive integer greater than or equal to $2$.  Following James, an $e$-abacus has $e$ vertical runners which are labelled $0,1, \ldots, e-1$. The positions on the
abacus are labelled, starting with $0$, from left to right and top to bottom.
If $\beta = (\beta_1, \beta_2, \dotsc, \beta_s)$ is a sequence of $\beta$-numbers for a partition $\lambda$, then for
each $1 \leq i \leq s$  we place a bead at position $\beta_i$ in order to obtain the $e$-abacus
display for $\lambda$ with $s$ beads.

Removing a rim hook of length $h$ from $\lambda$ corresponds in an abacus display of $\lambda$ to moving
a bead from position $a$ to a vacant position $b$ such that $a-b =h$. The leg-length
of the hook is given by the number of occupied positions between $a$ and $b$. When we slide all the beads as far up their respective runners as possible, we obtain the $e$-core of $\lambda$, and the total number of times
 we slide each bead up on their respective runners in doing so is the $e$-weight of $\lambda$.  The relative ($e$-)sign of $\lambda$, denoted by $\sigma_e(\lambda)$, can be defined as $(-1)^t$, where $t$ is the total leg-lengths of the $e$-hooks removed to obtain the $e$-core (see \cite[\S2]{MO}).

Let $\lambda'_j = |\{ i \mid \lambda_i \geq j \}|$ for all $j  \in {\mathbb Z}^+$, then
$\lambda'= (\lambda'_1, \lambda'_2, \dotsc)$ is the conjugate partition of $\lambda$. An abacus display of
$\lambda'$ can be obtained by rotating the abacus display of $\lambda$ through an angle of $\pi$, and by
reading the vacant positions as occupied and the occupied positions as vacant. Therefore, the $e$-core of
$\lambda'$ is the conjugate partition of the $e$-core of $\lambda$, and $\lambda'$ has the same $e$-weight
as $\lambda$.

The partition $\lambda$ is $e$-regular if whenever $\lambda_{i+1} =  \lambda_{i+2} =
\dotsb =\lambda_{i+j} >0$ then $j <e$, and $\lambda$ is $e$-restricted if $\lambda'$ is $e$-regular.   Mullineux
defined in~\cite{M} an involution $\lambda \mapsto m(\lambda)$ on the set of $e$-regular partitions of $n$.
This involution plays an important role in the representation theory of the symmetric groups and related algebras, and is closely
connected to the Alvis-Curtis duality for general linear groups, which we will describe in the next section.

The standard lexicographic and dominance order on $\mathcal{P}_n$ shall be denoted as $\geq$ and $\domeq$ respectively.

\subsection{Modular Alvis-Curtis duality for general linear groups}

Let $G = {\rm GL}_n(q)$ and let $(K,\O,k)$ be an $\ell$-modular system for some prime $\ell \nmid q$.

Denote by $T$ the maximal torus of invertible diagonal matrices in $G$,
by $U$ the group of upper unitriangular matrices
and set $B = UT$.
Let $W$ be the Weyl group of $G$. Then $W$ is isomorphic to the symmetric group $\mathfrak{S}_n$ on $n$ letters by identifying
$\mathfrak{S}_n$ with the subgroup of permutation matrices of
$G$.
For $\lambda \in \Lambda_n$,
denote by $\mathfrak{S}_\lambda$ the associated Young subgroup of
$\mathfrak{S}_n$ and by $Q_\lambda$ the standard parabolic subgroup
of $G$ generated by $B$ and $\mathfrak{S}_\lambda$. Denote by $U_\lambda$
the unipotent radical of $Q_\lambda$
and by $M_\lambda$ the standard Levi complement of $U_\lambda$ in
$Q_\lambda$.  Furthermore, let $e_{U_\lambda}= \frac{1}{|U_\lambda|} \sum_{u \in U_\lambda} u$.

Let $R \in \{K, \O, k \}$.  Alvis-Curtis character duality is defined as the duality operator on the Grothendieck group of $RG$ given by
$$ D_G(-)  =  \sum_{\lambda \in \Lambda(n)} (-1)^{n-l(\lambda)} RG e_{U_\lambda} \otimes_{RM_\lambda}
e_{U_\lambda} (-). $$

Following James \cite{J}, for each $\lambda \in \mathcal{P}_n$, we have a distinguished unipotent irreducible
$K{\rm GL}_{n}(q)$-module $S(1,\lambda)$ called a unipotent Specht module.
Note that the ordinary irreducible character of $S(1,\lambda)$
corresponds exactly to the unipotent character
$\chi_{1,\lambda}$ in the parametrization given by Green
\cite{G}.

As a $kG$-module, $S(1,\lambda)$ may not be irreducible.
However, it has a simple head
$L(1,\lambda)=S(1,\lambda) / \rad(S(1,\lambda))$.

It is well-known that the Alvis-Curtis dual of $S(1,\lambda)$ is $S(1,\lambda')$, i.e. $D_G([S(1,\lambda)]) = [S(1,\lambda')]$. But when we apply
$D_G$ to the simple modules $L(1,\lambda)$, all we can say is that
$D_G([L(1,\lambda)]) = \sum_{\mu \in{\mathcal P}_n} a_{\lambda\mu}[L(1,\mu)]$ for some integers $a_{\lambda\mu}$.
Ackermann and the first author \cite{AS} showed that there is a link between the
integers $a_{\lambda\mu}$ and the decomposition numbers of $G$. Namely,
define a $|{\mathcal P}_n| \times |{\mathcal P}_n|$ matrix
$\mathbf{A}_G = (a_{\lambda\mu})_{\lambda,\mu \in {\mathcal P}_n}$ where the rows and columns
are ordered in descending order with respect to the lexicographic order on $\mathcal{P}_n$.
Let $\mathbf{D}_G$ be the decomposition matrix of $G$ and denote by
$\mathbf{D}_u=([S(1,\lambda):L(1,\mu)])_{\lambda,\mu \in {\mathcal P}_n}$
the upper quadratic part of $\mathbf{D}_G$ corresponding the unipotent part. Let $\mathbf{P}$ be the permutation matrix given by the permutation on ${\mathcal P}_n$ sending
$\lambda$ to $\lambda'$. Then

\begin{thm}\cite[Theorem 3.2]{AS}\label{decompositionmatrixG}
With the notations above,
\begin{enumerate}
\item $\mathbf{A}_G = \mathbf{D}_u^{-1} \mathbf{P} \mathbf{D}_u$, in particular, $\mathbf{A}_G$ determines, and is determined by, $\mathbf{D}_u$.
\item The matrices $\mathbf{A}_{{\rm GL}_m(q^d)}$ arising from Alvis-Curtis duality of ${\rm GL}_m(q^d)$ for all $d$ and $m$ such that $dm \leq n$ determine $\mathbf{D}_G$.
\item The (entire) decomposition matrix of $G$ determines the Alvis-Curtis duality on all
 (that is unipotent and non-unipotent) irreducible $kG$-modules.
\end{enumerate}
\end{thm}

Furthermore, Alvis-Curtis duality is closely linked to tensoring with the sign representation of the symmetric
group. For the modular Alvis-Curtis duality this implies that in the columns corresponding to $e$-regular
partitions, there is only one non-zero entry which is equal to one and it is determined by the Mullineux map.

\begin{thm}[{\cite[Theorem 4.1]{AS}}] \label{Mullineux}
Let $\lambda$ and $\mu$ be two partitions of $n$, with $\mu$ $e$-regular.
Then
$$
a_{\lambda\mu}  =
\begin{cases}
1 & \text{if } \lambda = m(\mu); \\
0 & \text{otherwise.}
\end{cases}
$$
\end{thm}

\subsection{The $\q$-Schur algebra}

Let $F$ be a field of characteristic $p$, which can be zero or positive.  Let $\q \in F$ be a root of unity ($\q \ne 1$ if $p = 0$), and let $e$ be the least positive integer such that $1+\q+\dotsb+\q^{e-1} = 0$. Denote by $\mathcal{S}_n = \mathcal{S}_\q(n,n)$ the $\q$-Schur algebra (over $F$) as defined by Dipper and James in~\cite{DJ}. This has a distinguished class $\{ \Delta^{\mu} \mid \mu \in \mathcal{P}_n \}$ of modules called Weyl modules.  Each $\Delta^{\mu}$ has a simple head $L^{\mu}$ which is self-dual with respect to the contravariant duality induced by the anti-automorhpism of $\mathcal{S}_n$, and the set $\{ L^{\mu} \mid \mu \in \mathcal{P}_n \}$ is a complete set of mutually non-isomorphic simple modules of $\mathcal{S}_n$.  Note that $\Delta^{\mu}$ and $L^{\mu}$ are denoted as $W^{\lambda'}$ and $F^{\mu'}$ respectively in \cite{DJ2}.  The projective cover $P^\mu$ of $L^\mu$ (or of $\Delta^\mu$) has a filtration in which each factor is isomorphic to a Weyl module, and $\Delta^\mu$ occurs exactly once, at the top.  Furthermore, the multiplicity of $\Delta^{\lambda}$ in such a filtration is well-defined, and is equal to the multiplicity of $L^{\mu}$ as a composition factor of $\Delta^\lambda$.  We denote this multiplicity as $d_{\lambda\mu}$, which is a decomposition number of $\mathcal{S}_n$.

It is shown in \cite{DJ} that when $F = k$ and $q \cdot 1_F = \q$, the composition multiplicity of the simple module $L(1,\mu)$ in the Specht module $S(1,\lambda)$ of ${\rm GL}_n(q)$ equals $d_{\lambda\mu}$ for $\lambda,\mu \in \mathcal{P}_n$.

Denote by $c_{\lambda\mu}$ the composition multiplicity of $L^{\mu}$ in $P^{\lambda}$.  Then
$$c_{\lambda\mu} = \sum_{\nu \in \mathcal{P}_n} d_{\nu\lambda}d_{\nu\mu} = c_{\mu\lambda}.$$
Furthermore, define $e_{\lambda\mu}$ so that $\sum_{\nu \in \mathcal{P}_n} e_{\lambda\nu}d_{\nu\mu} = \delta_{\lambda\mu}$, and let $a_{\lambda\mu} = \sum_{\nu} e_{\lambda\nu}d_{\nu'\mu}$.  Thus when $F=k$ and $q \cdot 1_F = \q$, the integers $a_{\lambda\mu}$ are precisely the entries of the matrix $\mathbf{A}_{{\rm GL}_n(q)}$ by Theorem \ref{decompositionmatrixG}(1).

For $\lambda,\mu \in \mathcal{P}_n$, the Weyl modules $\Delta^\lambda$ and $\Delta^\mu$ lie in the same block of $\mathcal{S}_n$ if and only if $\lambda$ and $\mu$ have the same $e$-core, and hence the same $e$-weight as well.  Thus, $e$-core and $e$-weight are block invariants of $\q$-Schur algebras.  When $\Delta^{\lambda}$ lies in a block $B$ of $\mathcal{S}_n$, we say that $\lambda$ is a partition in $B$, and the $e$-core (resp.\ $e$-weight) of $\lambda$ is the $e$-core (resp.\ $e$-weight, or simply, weight) of $B$.

It is clear that $d_{\lambda\mu}$ depends on both $e$ and the characteristic $p$ of $F$.  Whenever the need to mention specifically what $e$ and $p$ are arises, we will write it as $d^{e,p}_{\lambda\mu}$.

The following are some well-known results of $\mathcal{S}_n$.

\begin{lem} \label{L:basic} \hfill
\begin{enumerate}
\item If $d_{\lambda\mu} \ne 0$, then $\lambda$ and $\mu$ have the same $e$-core and the same $e$-weight, and $\mu \domeq \lambda$.  Furthermore, $d_{\mu\mu} = 1$.

\item If $\mu$ is $e$-regular, then

\begin{enumerate}
\item $P^{\mu}$ is self-dual,

\item $L^{\mu}$ is the socle of $\Delta^{m(\mu)'}$,

\item $d_{\lambda\mu} = d_{\lambda'm(\mu)}$.
\end{enumerate}
In particular, $d_{m(\mu)'\mu} = 1$, and if $d_{\lambda\mu} \ne 0$, then $\mu \domeq \lambda \domeq m(\mu)'$.

\item $d_{\lambda\mu}^{e,0} \leq d_{\lambda\mu}^{e,p}$ for all primes $p$.
\end{enumerate}
\end{lem}

\begin{cor} \label{C:basice}
If $e_{\lambda\mu} \ne 0$, then $\lambda$ and $\mu$ have the same $e$-core and the same $e$-weight, and $\mu \domeq \lambda$,  and $e_{\mu\mu} = 1$.
\end{cor}

\begin{proof}
This follows from Lemma \ref{L:basic}(1) and the fact that $(e_{\lambda\mu})$ is the inverse matrix of $(d_{\lambda\mu})$.
\end{proof}

\subsection{The Jantzen filtration} \label{S:Jantzen}

Let $\Delta^\lambda = \Delta^\lambda(0) \supseteq \Delta^\lambda(1) \supseteq \Delta^\lambda(2) \supseteq \dotsb$
be the Jantzen filtration of the Weyl module $\Delta^\lambda$.  Thus $\Delta^\lambda(1) = \rad(\Delta^\lambda)$, and $\frac{\Delta^\lambda(i)}{\Delta^\lambda(i+1)}$ are self dual for all $i$.

Let $\lambda$ be a partition, and consider its $e$-abacus display, with $N$ beads say.  Suppose in moving a bead, say at position $a$, up its runner to some vacant position, say $a-ie$, we obtain (the abacus display of) a partition $\mu$.  Write $l_{\lambda \mu}$ for the number of occupied positions between $a$ and $a-ie$, and let $h_{\lambda\mu} = i$.  Also, write $\lambda \xrightarrow{\mu} \tau$ if the abacus display of $\mu$ with $N$ beads is also obtained from that of $\tau$ by moving a bead at position $b$ to $b-ie$, and $a < b$.  Thus if $\lambda \xrightarrow{\mu} \tau$, then the abacus display of $\tau$ with $N$ beads may be obtained from $\lambda$ in two steps: first move the bead at position $a$ to position $a-ie$ (which yields the abacus display of $\mu$), and then move the bead at position $b-ie$ to position $b$.

When $\lambda$ and $\mu$ are distinct partitions having the same $e$-core and the same $e$-weight, let
$$J_{\lambda\mu} = \sum (-1)^{l_{\lambda\rho} + l_{\tau\rho} + 1} (1+\nu_p(h_{\lambda \rho})) d_{\tau\mu},$$
where the sum runs through all $\tau$ and $\rho$ such that $\lambda \xrightarrow{\rho} \tau$, and where $\nu_p$ denotes the standard $p$-valuation if $p > 0$ and $\nu_0(x) = 0$ for all $x$.

Just as with $d_{\lambda\mu}$, we will write $J_{\lambda\mu}$ as $J_{\lambda\mu}^{e,p}$ when the need to specify what $e$ and $p$ are arises.


\begin{thm}[{see, for example, \cite[5.32]{MathasBook}}] \label{T:Jantzen}
Let $\lambda$ and $\mu$ be distinct partitions having the same $e$-core and the same $e$-weight.  Then
$$
J_{\lambda\mu} = \sum_{i\in \mathbb{Z}^+} [\Delta^\lambda(i) : L^\mu].
$$
In particular, $d_{\lambda \mu} \leq J_{\lambda\mu}$, and $d_{\lambda \mu} = 0$ if and only if $J_{\lambda\mu} = 0$.
\end{thm}

\subsection{Restriction and induction}

Given $m, n \in \mathbb{Z}^+$ with $m < n$, the $\q$-Schur algebra $\mathcal{S}_{m}$ can be naturally embedded into $\mathcal{S}_n$ and we thus have the restriction and induction functors, denoted by $\Res_{\mathcal{S}_{m}}^{\mathcal{S}_n} (-)$ and $\Ind_{\mathcal{S}_{m}}^{\mathcal{S}_n} (-)$, between the two module categories.  The effect of these functors on the image of the Weyl modules in the Grothendieck group can be easily described when $m = n-1$:

\begin{thm}
Let $\lambda \in \mathcal{P}_{n-1}$ and $\mu \in \mathcal{P}_n$.  Then
\begin{align*}
[\Ind_{\mathcal{S}_{n-1}}^{\mathcal{S}_n} (\Delta^{\lambda})] &= \sum_{\rho} [\Delta^{\rho}], \\
[\Res_{\mathcal{S}_{n-1}}^{\mathcal{S}_n} (\Delta^{\mu})] &= \sum_{\tau} [\Delta^{\tau}],
\end{align*}
where the sums run over all partitions $\rho$ obtained from the abacus display of $\lambda$ by moving a bead to its vacant succeeding position, and over all partitions $\tau$ obtained from the abacus display of $\mu$ by moving a bead to its vacant preceding position respectively.
\end{thm}

When $B$ is a block of $\mathcal{S}_n$, $C$ is a block of $\mathcal{S}_{m}$, $M$ is a $\mathcal{S}_n$-module and $N$ is a $\mathcal{S}_m$-module, we write $M \down_{C}$ for the projection of $\Res_{\mathcal{S}_m}^{\mathcal{S}_n} (M)$ onto $C$, and $N \up^{B}$ for the projection of $\Ind_{\mathcal{S}_m}^{\mathcal{S}_n} (N)$ onto $B$.

\subsection{$[w:k]$-pairs}

Let $B$ be a block of $\mathcal{S}_n$, with $e$-core $\kappa$ and weight $w$.  Suppose that in an abacus display of $\kappa_B$, runner $i$ has $k$ beads more than runner $(i-1)$ for some $1 \leq i < e$.  Let $C$ be the weight $w$ block of $\mathcal{S}_{n-k}$ whose $e$-core $\kappa_C$ can be obtained from the abacus display of $\kappa_B$ by interchanging runners $(i-1)$ and $i$.  The blocks $B$ and $C$ are said to form a $[w:k]$-pair.

Every partition in $B$ has at least $k$ normal  beads on runner $i$, while every partition in $C$ has at least $k$ conormal beads on runner $(i-1)$ (for the definition of normal and conormal beads, see, for example, \cite[Section 1.1.2]{F1}).  Given a partition $\lambda$ in $B$, let $\Phi(\lambda) = \Phi_{B,C}(\lambda)$ denote the partition in $C$ obtained from $\lambda$ by moving the $k$ topmost normal beads on runner $i$ to their respective preceding positions on runner $(i-1)$.  Then $\Phi$ is a bijection from the set of partitions in $B$ to the set of partitions in $C$.  Furthermore, We have the following:

\begin{thm}[\cite{Bru}] \label{T:Bru}
Let $\lambda$ be a partition in $B$. Then
\begin{enumerate}
\item $\lambda$ is $e$-regular if and only if $\Phi(\lambda)$ is $e$-regular;
\item $\soc(L^\lambda \down_C) = (L^{\Phi(\lambda)})^{\oplus k!}$;
\item $\soc(L^{\Phi(\lambda)} \up^B) = (L^\lambda)^{\oplus k!}$;
\item $\lambda$ has exactly $k$ normal beads on runner $i$ if and only if $\Phi(\lambda)$ has exactly $k$ conormal beads on runner $(i-1)$, in which case, $L^\lambda \down_C = (L^{\Phi(\lambda)})^{\oplus k!}$ and $L^{\Phi(\lambda)} \up^B = (L^\lambda)^{\oplus k!}$.
\end{enumerate}
\end{thm}

Note that for a partition in $B$, the following statements are equivalent:
\begin{itemize}
\item It has exactly $k$ beads on runner $i$ whose respective preceding positions are vacant.
\item The respective succeeding positions of its beads on runner $(i-1)$ are all occupied.
\end{itemize}
When these statements hold, the effect of $\Phi$ is merely to interchange the runners $(i-1)$ and $i$.  These statements hold for all partitions in $B$ if and only if $w \leq k$.

Following Fayers \cite{Fayers}, we say that $B$ and $C$ are {\em Scopes equivalent} when $w \leq k$, and further extend Scopes equivalence reflexively and transitively to an equivalence relation on the set of weight $w$ blocks of $\q$-Schur algebras.

\subsection{The Fock space representation of $U_v(\widehat{\mathfrak{sl}}_e)$}

The Fock space representation $\mathcal{F}$ of $U_v(\widehat{\mathfrak{sl}}_e)$ has a basis $\{s(\lambda) \mid \lambda \in \mathcal{P} \}$ as a vector space over $\mathbb{C}(v)$.  The canonical bases $\{G(\lambda) \mid \lambda \in \mathcal{P} \}$ and $\{G^-(\lambda) \mid \lambda \in \mathcal{P} \}$ of $\mathcal{F}$ can be characterised as follows:
\begin{alignat*}{2}
G(\lambda) - s(\lambda) &\in \bigoplus_{\mu \in \mathcal{P}} v\mathbb{Z}[v] s(\mu);
\qquad \quad &G^-(\lambda) - s(\lambda) & \in \bigoplus_{\mu \in \mathcal{P}} v^{-1}\mathbb{Z}[v^{-1}] s(\mu); \\
\overline{G(\lambda)} &= G(\lambda); & \overline{G^-(\lambda)} &= G^-(\lambda).
\end{alignat*}
Here, $x \mapsto \overline{x}$ is the involution on $\mathcal{F}$ defined by Leclerc and Thibon in \cite{LT}.

Let $\left< - , - \right>$ be the inner product on $\mathcal{F}$ with respect to which $\{s(\lambda) \mid \lambda \in \mathcal{P} \}$ is orthonormal.  Define $d_{\lambda\mu}(v)$ and $e_{\lambda\mu}(v)$ as follows:
$$
d_{\lambda\mu}(v) = \left< G(\mu),s(\lambda) \right>, \qquad
e_{\lambda\mu}(-v^{-1}) = \left< G^-(\lambda),s(\mu) \right>.
$$

There are occasions where we need to consider $d_{\lambda\mu}(v)$ arising from the Fock space representations of both $U_v(\widehat{\mathfrak{sl}}_e)$ and $U_v(\widehat{\mathfrak{sl}}_{e'})$, and when these happen, we shall write $d_{\lambda\mu}^e(v)$ and $d_{\lambda\mu}^{e'}(v)$ as appropriate.

We collate together some well-known properties of $d_{\lambda\mu}(v)$ and $e_{\lambda\mu}(v)$.

\begin{thm} \label{T:vdecomp} \hfill
\begin{enumerate}
\item $d_{\mu\mu}(v) = 1 = e_{\mu\mu}(v)$;
\item $d_{\lambda\mu}(v), e_{\lambda\mu}(v) \in v\mathbb{N}_0[v]$ if $\lambda \ne \mu$;
\item $\sum_{\nu} d_{\lambda\nu}(v)e_{\nu',\mu'}(-v) = \delta_{\lambda\mu} = \sum_{\nu} e_{\lambda'\nu'}(-v)d_{\nu\mu}(v)$;
\item $d^e_{\lambda\mu}(1) = d^{e,0}_{\lambda\mu}$;
\end{enumerate}
\end{thm}

\begin{proof}
(1) follows from \cite[7.2]{L}, (2) and (4) are proved by Varagnolo and Vasserot \cite{VV}, and (3) is Theorem 12 of \cite{L}.
\end{proof}

\begin{thm} \label{T:Mull}
Let $\mu$ be an $e$-regular partition having $e$-weight $w$.
\begin{enumerate}
\item $d_{\lambda\mu}(v) = v^w d_{\lambda'm(\mu)}(v^{-1})$; in particular, $d_{m(\mu)'\mu}(v) = v^w$.
\item If $d_{\lambda\mu}(v) \ne 0$ and $\lambda \ne m(\mu)'$, then $\deg(d_{\lambda\mu}(v)) < w$.
\end{enumerate}
\end{thm}

\begin{proof}
These are Theorem 7.2 and Corollary 7.7 of \cite{LLT} respectively.
\end{proof}

 Given an $e$-abacus display of a partition $\lambda$, we can insert a new runner, whose topmost $k$ positions are occupied while the remainder are vacant, either between two consecutive runners, or to the left of runner $0$, or to the right of runner $(e-1)$, and obtain the $(e+1)$-abacus display of a new partition $\widehat{\lambda}$.  The runner which we insert is said to be {\em empty} (relative to $\lambda$) if the topmost $k$ positions in each runner of the abacus display of $\lambda$ are all occupied, and {\em full} (relative to $\lambda$) if all the beads in each runner of the abacus display of $\lambda$ occur in the topmost $k$ positions.  We note that if $\widehat{\lambda}$ is obtained from $\lambda$ by inserting an empty runner, then $\widehat{\lambda}$ is always $(e+1)$-regular, while if $\widehat{\lambda}$ is obtained from $\lambda$ by inserting a full runner, then $\widehat{\lambda}$ is $(e+1)$-regular if and only if $\lambda$ is $e$-regular.

James and Mathas related $d_{\lambda\mu}(v)$ arising from different Fock spaces by showing that $d_{\lambda\mu}(v)$ remains invariant under the insertion of empty runners.

\begin{thm}[{\cite[Theorem 4.5]{JM}}] \label{T:equating}
Let $\lambda$ and $\mu$ be partitions having the same $e$-weight and the same $e$-core, and display them on an $e$-abacus with $t$ beads, for some large enough $t$.  Let $i$ ($0 \leq i \leq e$) be a fixed integer, and let $\lambda^+$ and $\mu^+$ be the partitions obtained by inserting the same runner, which is empty relative to both $\lambda$ and $\mu$, between runners $(i-1)$ and $i$ of the $e$-abacus display of $\lambda$ (resp.\ $\mu$).  Then
$$
d^e_{\lambda\mu}(v) = d^{e+1}_{\lambda^+\mu^+}(v).
$$
\end{thm}

We need a result dual to Theorem \ref{T:equating} involving the insertion of a full runner.

\begin{thm} \label{T:full}
Let $\lambda$ and $\mu$ be partitions having the same $e$-weight and the same $e$-core, and let $\widehat{\lambda}$ and $\widehat{\mu}$ be the partitions obtained from $\lambda$ and $\mu$ respectively by inserting the same runner, which is full relative to both $\lambda$ and $\mu$.  Then
$$
d^e_{\lambda\mu}(v) = d^{e+1}_{\widehat{\lambda}\widehat{\mu}}(v).
$$
\end{thm}

Theorem \ref{T:full} is the main result of \cite{F2}, where it is proved directly for the case when $\mu$ is $e$-regular, and then uses Theorem \ref{T:equating} to deal with the general case.  We provide here an independent proof that the case when $\mu$ is $e$-regular also follows from Theorems \ref{T:equating} and \ref{T:Mull}.  Our proof also provides an important corollary which we shall require.

\begin{proof}
Recall that to obtain an abacus display of the conjugate of a partition, one may take an abacus display of the partition, rotate it through an angle of $\pi$, and read the occupied positions as vacant and the vacant positions as occupied.  As such, inserting a full runner to a partition is equivalent to first conjugating the partition, followed by inserting an empty runner and then conjugating the resultant partition.  Thus, $\widehat{\lambda} = ((\lambda')^+)'$, and $\widehat{\mu} = ((\mu')^+)'$; here, and hereafter, when $\nu$ is a partition having the same $e$-weight and $e$-core as $\lambda'$ (or $\mu'$), $\nu^+$ denotes the partition obtained from $\nu$ by inserting a (fixed) runner which is empty relative to both $\lambda'$ and $\mu'$.

Let the $e$-weights of $\lambda$ and $\mu$ be $w$, and assume that $\mu$ is $e$-regular.  Then $\mu \trianglerighteq m(\mu)'$ by Lemma \ref{L:basic}(1,2), so that $\mu' \trianglelefteq m(\mu)$.  Thus, as the runner inserted into $\mu'$ to obtain $(\mu')^+$ is empty relative to $\mu'$, it is also empty relative to $m(\mu)$.  Hence by Theorems \ref{T:Mull}(1) and \ref{T:equating}, we have
\begin{align*}
d^e_{\lambda\mu}(v) &= v^w d^e_{\lambda'm(\mu)}(v^{-1}) \\
&= v^w d^{e+1}_{(\lambda')^+ m(\mu)^+}(v^{-1})\\
&= d^{e+1}_{((\lambda')^+)' m(m(\mu)^+)}(v) \\
&= d^{e+1}_{ \widehat{\lambda} m(m(\mu)^+)}(v),
\end{align*}
and
$$
v^w = d^e_{\mu'm(\mu)}(v) = d^{e+1}_{(\mu')^+ m(\mu)^+}(v).
$$
The latter yields $m(m(\mu)^+)' = (\mu')^+$ by Theorem \ref{T:Mull}, so that $\widehat{\mu} = ((\mu')^+)' = m(m(\mu)^+)$, and substituting this into the former yields
$$
d^e_{\lambda\mu}(v) = d^{e+1}_{\widehat{\lambda}\widehat{\mu}}(v).
$$

Now, for general $\mu$, to each of $\lambda$ and $\mu$, we can first insert an empty runner (to $e$-regularise $\mu$), then insert a full runner, and finally removing the empty runner we have inserted and obtain the following:
\begin{align*}
d^e_{\lambda \mu}(v) &= d^{e+1}_{\lambda^+\mu^+}(v) \\
&= d^{e+2}_{\widehat{\lambda^+}\widehat{\mu^+}}(v) \\
&= d^{e+1}_{\widehat{\lambda}\widehat{\mu}} (v).
\end{align*}
\end{proof}

\begin{cor}[of proof] \label{C:full}
Let $\mu$ be an $e$-regular partition, and let $\widehat{\mu}$ be the partition obtained from $\mu$ by inserting a full runner.  Then this runner is also full relative to $m(\mu)'$, and denoting the partition obtained by inserting this runner to $m(\mu)'$ as $\widehat{m(\mu)'}$, we have $\widehat{m(\mu)'} = m(\widehat{\mu})'$.
\end{cor}

\begin{proof}
We have seen that $\widehat{\mu} = m(m(\mu)^+)$.  Thus
$$ m(\widehat{\mu})' = (m(\mu)^+)' = \widehat{m(\mu)'}.$$
\end{proof}

\begin{thm} \label{T:R-H}
$\frac{d}{dv} (d^e_{\lambda\mu}(v) )|_{v=1} = J^{e,0}_{\lambda\mu}$.
\end{thm}

\begin{proof}
This is proved by Ryom-Hansen \cite{R-H} for $\mu$ $e$-regular.  In addition, the Theorem holds trivially when $d^e_{\lambda\mu}(v) =0$ by Theorems \ref{T:vdecomp}(4) and \ref{T:Jantzen}.  It remains to consider the case when $\mu$ is $e$-singular and $d^e_{\lambda\mu}(v) \ne 0$.  Given a partition $\nu$ with the same $e$-core and the same $e$-weight as $\lambda$, write $\nu^+$ for the partition obtained from $\nu$ by inserting a (fixed) runner which is empty relative to $\lambda$.  Since $d^e_{\lambda\mu}(v) \ne 0$, we have $\mu \domeq \lambda$ by Theorem \ref{T:vdecomp}(2,4) and Lemma \ref{L:basic}(1), so that the runner inserted into $\mu$ to obtain $\mu^+$ is also empty relative to $\mu$.  Thus, by Theorem \ref{T:equating}, and since $\mu^+$ is $(e+1)$-regular, we have
$$
\tfrac{d}{dv} (d^e_{\lambda\mu}(v) )|_{v=1} =
\tfrac{d}{dv} (d^{e+1}_{\lambda^+\mu^+}(v) )|_{v=1} =
J_{\lambda^+\mu^+}^{e+1,0}.
$$
It remains to show that $J_{\lambda^+\mu^+}^{e+1,0} = J_{\lambda\mu}^{e,0}$.  This holds because
\begin{itemize}
\item $\lambda \xrightarrow{\rho} \tau$ if and only if $\lambda^+ \xrightarrow{\rho^+} \tau^+$;
\item if $\lambda \xrightarrow{\rho} \tau$, then
\begin{itemize}
\item[$\diamond$] $\tau \dom \lambda$ so that the runner inserted into $\tau$ to obtain $\tau^+$ is empty relative to $\tau$, and hence $d^{e,0}_{\tau\mu} = d^{e+1,0}_{\tau^+\mu^+}$ by Theorems \ref{T:equating} and \ref{T:vdecomp}(4);
\item[$\diamond$] $l_{\lambda\rho} = l_{\lambda^+\rho^+}$ and $l_{\tau\rho} = l_{\tau^+\rho^+}$;
\end{itemize}
\item $\lambda^+ \xrightarrow{\sigma} \upsilon$ only if $\sigma = \rho^+$ and $\upsilon = \tau^+$ for some $\rho$ and $\tau$.
\end{itemize}
\end{proof}

\begin{prop} \label{P:degree}
Suppose $d_{\lambda\mu}^e(v) \ne 0$.  Then $\deg(d_{\lambda\mu}^e(v)) \leq w$.
\end{prop}

\begin{proof}
When $\mu$ is $e$-regular, this follows from Theorem \ref{T:Mull}.  When $\mu$ is $e$-singular, let $\lambda^+$ and $\mu^+$ be the partitions as defined in Theorem \ref{T:equating}.  Then $d^e_{\lambda\mu}(v) = d^{e+1}_{\lambda^+\mu^+}(v)$.  Since $\mu^+$ is $(e+1)$-regular, and has the same weight as $\mu$, we see that $\deg (d^{e+1}_{\lambda^+\mu^+}(v)) \leq w$, and the Proposition follows.
\end{proof}

Theorem \ref{T:vdecomp}(2) can be strengthened as follows:

\begin{thm}[{\cite[Theorem 2.4]{T2}}] \label{T:parity}
Suppose $d_{\lambda\mu}^e(v) \ne 0$.  Then
$$
d^e_{\lambda\mu}(v) \in
\begin{cases}
\mathbb{N}_0[v^2], &\text{if }\sigma_e(\lambda) = \sigma_e(\mu); \\
v\mathbb{N}_0[v^2], &\text{otherwise.}
\end{cases}
$$
\end{thm}

\subsection{Rouquier blocks}

Consider the blocks of the $\q$-Schur algebras of a fixed weight $w$ whose $e$-cores have the following property:  on their abacus displays, either runner $i$ has at least $w$ beads more than runner $j$, or runner $j$ has at least $(w-1)$ beads more than runner $i$, for all $0 \leq i< j < e$.  These blocks, known as Rouquier blocks, form a single Scopes equivalence class, and are now well understood.   In particular, there exist closed formulas for $d_{\lambda\mu}(v)$ and $e_{\lambda\mu}(v)$ for $\mu$ lying in a Rouquier block (see \cite{LM} and \cite{T1}).  This thus gives the decomposition numbers of these blocks of $\q$-Schur algebras in characteristic $0$ upon evaluation at $v=1$.  In fact, this also gives the decomposition numbers in the `Abelian defect' case, i.e.\ the case where $p> w$ (see \cite{JLM}).

An arbitrary weight $w$ block can always be induced to a Rouquier block through a sequence of $[w:k]$-pairs.

\begin{lem} \label{L:sequence}
Suppose $A$ is a weight $w$ block of $\mathcal{S}_n$.  Then there exists a sequence $B_0, B_1,\dotsc, B_s$ of weight $w$ blocks of $\q$-Schur algebras such that $B_0 = A$, $B_s$ is Rouquier, and for each $1 \leq i \leq s$, the blocks $B_{i}$ and $B_{i-1}$ form a $[w:k_i]$-pair for some $k_i \in \mathbb{Z}^+$.
\end{lem}

\begin{proof}
This is Lemma 3.1 of \cite{Fayers} in the context of the Iwahori-Hecke algebras, and its proof can be adapted for $\q$-Schur algebras.
\end{proof}

We call the weight $w$ block which has an abacus display in which runner $i$ has exactly $(w-1)$ beads more than runner $i-1$ for all $1 \leq i <e$ the {\em canonical Rouquier block}.  The partitions in this block are called {\em canonical Rouquier partitions}.

\section{Weight $2$ blocks}

Weight $2$ blocks of symmetric groups and Schur algebras in odd characteristic are well understood by the work of several authors (see, for example, \cite{S}, \cite{R}, \cite{CT}).  In this section, we show that many results of \cite{S} and \cite{CT} can be generalised to weight $2$ blocks of $\q$-Schur algebras, as long as the underlying characteristic is not $2$.

We begin by introducing some notations relating to weight $2$ partitions due to Richards \cite{R}.  If $\lambda$ is such a partition, we denote by $\del\lambda$ the absolute difference between the leg-lengths of the $e$-hooks removed from $\lambda$ to obtained its $e$-core.  Furthermore, if $\del\lambda = 0$, we say $\lambda$ is black if $\lambda$ has two $e$-hooks and the larger leg-length is even, or $\lambda$ has one $e$-hook and one $(2e)$-hook and the leg-length of the $(2e)$-hook is congruent to $0$ or $3$ modulo 4; otherwise, $\lambda$ is white.  We note that the relative $e$-sign of $\lambda$ is the parity of $\del\lambda$, i.e.\ $\sigma_e(\lambda) = (-1)^{\del\lambda}$.

\begin{thm}\cite[Lemmas 4.2 and 4.3, Theorem 4.4]{R}\label{T:R}
Consider the set $A$ of all weight $2$ partitions having a given $e$-core.  For each $0 \leq i < e$, let $A_i = \{ \lambda \in A \mid \del \lambda = i \}$.  In addition, let $A_{0,b} = \{ \lambda \in A_0 \mid \lambda \text{ is black} \}$, and $A_{0,w} = \{ \lambda \in A_0 \mid \lambda \text{ is white} \}$.
\begin{enumerate}
\item For each $0 \leq i < e$, $A_i$ is totally ordered by $\domeq$.
\item Let $B \in \{A_i \mid 1 \leq i < e\} \cup \{A_{0,b}, A_{0,w} \}$.
\begin{enumerate}
\item $B$ is non-empty.
\item $\lambda \in B$ is $e$-singular if and only if $\lambda$ is the least partition in $B$ (with respect to $\domeq$).
\item If $\lambda \in B$ is $e$-regular, then $m(\lambda)'$ is the next smaller partition in $B$ (with respect to $\domeq$).
\end{enumerate}
\end{enumerate}
\end{thm}

\begin{prop} \label{P:radWeyl}
Let $\lambda$ be a weight $2$ partition.  If $L^{\mu}$ and $L^{\nu}$ are two non-isomorphic composition factors of $\rad(\Delta^\lambda)$, then $\del \mu \ne \del\nu$ unless $\del\mu = \del\nu =0$, in which case, $\mu$ and $\nu$ are of different colour.
\end{prop}

\begin{proof}
Suppose first that $\del\mu = \del\nu > 1$.  We may assume that $\mu \dom \nu$, and hence that $\mu$ is $e$-regular, by Theorem \ref{T:R}(1,2(b)).  Since $d_{\lambda\mu}, d_{\lambda\nu} \ne 0$, and $\lambda \notin \{\mu,\nu\}$, we have $\mu \dom \lambda \domeq m(\mu)'$ and $\nu \dom \lambda$ by Lemma \ref{L:basic}(1,2).  But, as $\mu \dom \nu$, we have $m(\mu)' \domeq \nu$ by Theorem \ref{T:R}(1,2(c)), so that
$$
\mu \dom \lambda \domeq m(\mu)' \domeq \nu \dom \lambda,
$$
a contradiction.  A similar argument shows that we also cannot have $\del\mu = \del\nu =0$ and at the same time $\mu$ and $\nu$ having the same colour.
\end{proof}

We remark that there is no restriction on the values of $p$ in Proposition \ref{P:radWeyl}.

\begin{prop} \label{P:decomp}
Suppose that $p \ne 2$.  Let $\lambda$ and $\mu$ be two partitions having $e$-weight $2$ and the same $e$-core, and assume that $\lambda \notin \{\mu,m(\mu)'\}$.
\begin{enumerate}
\item If $d^e_{\lambda\mu}(v) \ne 0$, then $d^e_{\lambda\mu}(v) = v$.
\item $d^e_{\lambda\mu}(1) = d^{e,p}_{\lambda\mu} = J^{e,p}_{\lambda\mu}$.
\end{enumerate}
\end{prop}

\begin{proof} \hfill
\begin{enumerate}
\item Suppose first that $e$ is an odd prime.  By Theorem 2.2 of \cite{CT}, we have $d^{e,e}_{\lambda\mu} \ne 0$ if and only if $d^{e,e}_{\lambda\mu} = 1$, in which case $\sigma_e(\lambda) \ne \sigma_e(\mu)$ unless $\lambda\in \{ \mu, m(\mu)' \}$.  Thus, if $d^e_{\lambda\mu}(v) \ne 0$, then $0 <  d^e_{\lambda\mu}(1) = d^{e,0}_{\lambda\mu} \leq d^{e,e}_{\lambda\mu} \leq 1$ by Theorem \ref{T:vdecomp}(4) and Lemma \ref{L:basic}(3), so that $d^e_{\lambda\mu}(1) = d^{e,e}_{\lambda\mu} = 1$.  Hence $d^e_{\lambda\mu}(v)$ is a non-constant monic monomial by Theorem \ref{T:vdecomp}(2), whose degree is bounded above by $2$ by Proposition \ref{P:degree}.  Furthermore, since $\lambda \notin \{\mu,m(\mu)'\}$, we see that $\sigma_e(\lambda) \ne \sigma_e(\mu)$, so that $d_{\lambda\mu}(v) = v$ by Theorem \ref{T:parity}.

If $e = 2$ or $e$ is composite, let $k$ be a positive integer such that $e'= e + k$ is an odd prime.  By inserting $k$ full runners to $\lambda$ and $\mu$ to obtain $\widehat{\lambda}$ and $\widehat{\mu}$ respectively, we have $d^e_{\lambda\mu}(v) = d^{e'}_{\widehat{\lambda}\widehat{\mu}}(v)$ by Theorem \ref{T:full}.  Furthermore, $\lambda \notin \{\mu,m(\mu)'\}$ implies that $\widehat{\lambda} \notin \{\widehat{\mu},\widehat{m(\mu)'} \} = \{\widehat{\mu},m(\widehat{\mu})'\}$ by Corollary \ref{C:full}, so that (1) follows.

\item We prove by induction on $\lambda$.  If $\lambda$ is a maximal with respect to $\dom$, and $\lambda \ne \mu$, then clearly $d^e_{\lambda\mu}(1) = d^{e,p}_{\lambda\mu} = J^{e,p}_{\lambda\mu} = 0$.  Assume that $d^e_{\nu\mu}(1) = d^{e,p}_{\nu\mu} = J^{e,p}_{\nu\mu}$ whenever $\nu \dom \lambda$ and $\nu \notin \{ \mu, m(\mu)' \}$.  Then $d^e_{\nu\mu}(1) = d^{e,p}_{\nu\mu}$ for {\em all} $\nu \dom \lambda$ (even when $\nu \in \{\mu, m(\mu)'\}$).  This implies that $J^{e,0}_{\lambda\mu} = J^{e,p}_{\lambda\mu}$ for all $\lambda$ from the definition of $J_{\lambda\mu}^{e,-}$ (see Theorem \ref{T:Jantzen}).  Now, if $\lambda \notin \{\mu,m(\mu)'\}$, then by (1) and Theorem \ref{T:R-H}, we have $J^{e,0}_{\lambda\mu} \leq 1$.  This gives $d^{e,0}_{\lambda\mu} = J^{e,0}_{\lambda\mu}$ by Theorem \ref{T:Jantzen}.  Similarly, since $J^{e,p}_{\lambda\mu} = J^{e,0}_{\lambda\mu}$, we also have $d^{e,p}_{\lambda\mu} = J^{e,p}_{\lambda\mu}$.  Hence (2) follows.
\end{enumerate}
\end{proof}

\begin{cor}[of proof] \label{C:J}
Suppose that $p \ne 2$.  Let $\lambda$ and $\mu$ be two partitions having $e$-weight $2$ and the same $e$-core.  Then $$ J^{e,p}_{\lambda\mu} = J^{e,0}_{\lambda\mu}.$$
\end{cor}

\begin{thm} \label{T:wt2vdecomp}
Let $\lambda$ and $\mu$ be partitions having $e$-weight $2$ and the same $e$-core.  Then
$$
d_{\lambda\mu}^e(v) =
\begin{cases}
1, &\text{if } \lambda = \mu;\\
v, &\text{if } \mu \dom \lambda \dom m(\mu)', \text{ and } |\del\lambda - \del\mu| = 1; \\
v^2, &\text{if $\mu$ is $e$-regular, and } \lambda = m(\mu)'; \\
0, &\text{otherwise.}
\end{cases}
$$
The condition $\lambda \dom m(\mu)'$ is to be read as vacuous if $\mu$ is $e$-singular.
\end{thm}

\begin{proof}
This follows from Proposition \ref{P:decomp}, \cite[Theorem 2.2]{CT} and Theorem \ref{T:Mull}(1) when $e$ is an odd prime.  When $e=2$ or $e$ is composite, let $k$ be a positive integer such that $e' = e+k$ is an odd prime.  Given a weight $2$ partition $\tau$ having the same $e$-core as $\lambda$ and $\mu$, let $\widehat{\tau}$ be the partition obtained by inserting $k$ full runners to the abacus display of $\tau$.  Then $\del\widehat{\tau} = \del\tau$, and $\tau$ is $e$-regular if and only if $\widehat{\tau}$ is $e'$-regular, in which case $\widehat{m(\tau)'} = m(\widehat{\tau})'$ by Corollary \ref{C:full}.  Furthermore, if $\rho$ is another weight 2 partition with the $e$-same core as $\tau$, then $\rho \dom \tau$ if and only if $\widehat{\rho} \dom \widehat{\tau}$.  Thus,
\begin{align*}
d_{\lambda\mu}^e(v) = d_{\widehat{\lambda}\widehat{\mu}}^{e'}(v) &=
\begin{cases}
1, &\text{if } \widehat{\lambda} = \widehat{\mu};\\
v, &\text{if } \widehat{\mu} \dom \widehat{\lambda} \dom m(\widehat{\mu})', \text{ and } |\del\widehat{\lambda} - \del\widehat{\mu}| = 1; \\
v^2, &\text{if $\widehat{\mu}$ is $e'$-regular, and } \widehat{\lambda} = m(\widehat{\mu})'; \\
0, &\text{otherwise,}
\end{cases} \\
&=
\begin{cases}
1, &\text{if } \lambda = \mu;\\
v, &\text{if } \mu \dom \lambda \dom m(\mu)', \text{ and } |\del\lambda - \del\mu| = 1; \\
v^2, &\text{if $\mu$ is $e$-regular, and } \lambda = m(\mu)'; \\
0, &\text{otherwise.}
\end{cases}
\end{align*}
\end{proof}

\begin{cor} \label{C:decomp}
Suppose that $p \ne 2$.  Let $\lambda$ and $\mu$ be partitions in a weight $2$ block of a $\q$-Schur algebra.  Then
$$
d^{e,p}_{\lambda\mu} = d^e_{\lambda\mu}(1) =
\begin{cases}
1, &\text{if } \lambda \in \{\mu, m(\mu)'\}, \\
   &\text{or both } \mu \dom \lambda \dom m(\mu)' \text{ and } |\del\lambda - \del\mu| = 1; \\
0, &\text{otherwise.}
\end{cases}
$$
The condition $\lambda \dom m(\mu)'$ is to be read as vacuous if $\mu$ is $e$-singular.
\end{cor}

\begin{proof}
This follows from Theorem \ref{T:wt2vdecomp} and Proposition \ref{P:decomp}.
\end{proof}

\begin{cor} \label{C:cl}
Suppose that $p \ne 2$.  Let $\lambda$ be a partition in a weight $2$ block of a $\q$-Schur algebra.
For each $i = 0, 1, \dotsc, e-1$, let $n_i$ be the number of composition factors of $\Delta^\lambda$ which are labelled by partitions having $\del$-value $i$.  Then
\begin{enumerate}
\item $n_i = 0$ for all $i$ such that $|i - \del\lambda| > 1$;
\item $n_{\del\lambda+1} \leq 1$;
\item $n_{\del\lambda} \leq 2$, with equality if and only if $\lambda$ is $e$-regular;
\item If $n_{\del\lambda-1} >1$, then $\del\lambda=1$ and $n_{\del\lambda-1} = 2$, and the two partitions which label the composition factors of $\Delta^{\lambda}$ and which have $\del$-value $0$ are of different colour.
\end{enumerate}
In particular, the composition length of $\Delta^\lambda$ is at most $5$, with equality only if $\lambda$ is $e$-regular and $\del\lambda =1$.
\end{cor}

\begin{proof}
This follows from Theorem \ref{T:R}(2(c)), Proposition \ref{P:radWeyl} and Corollary \ref{C:decomp}.
\end{proof}

We now describe the Ext-quivers, and the Jantzen filtration and the radical filtration of the Weyl modules, of weight $2$ blocks of $\q$-Schur algebras when $p \ne 2$. Let $\lambda$ be a partition with $e$-weight $2$, and let
$$
\Delta^\lambda = \Delta^\lambda(0) \supseteq \Delta^\lambda(1) \supseteq \Delta^\lambda(2) \supseteq \dotsb
$$
be the Jantzen filtration of the Weyl module $\Delta^\lambda$.  For a partition $\mu$ having the same $e$-core and the same $e$-weigtht as $\lambda$, let $\operatorname{Jan}_{\lambda\mu}(v) = \sum_{i \geq 0} [\frac{\Delta^\lambda(i)}{\Delta^\lambda(i+1)}: L^\mu] v^i$, so that as $\mu$ varies, $\operatorname{Jan}_{\lambda\mu}(v)$ describes the Jantzen layers of $\Delta^\lambda$.  Similarly, let $\rad_{\lambda\mu}(v) = \sum_{i \geq 0} [\frac{\rad^i(\Delta^\lambda)}{\rad^{i+1}(\Delta^\lambda)} : L^{\mu}] v^i$, so that as $\mu$ varies, $\rad_{\lambda\mu}(v)$ describes the radical layers of $\Delta^\lambda$.

\begin{thm} \label{T:ext}
Suppose that $p \ne 2$.  Let $\lambda$ and $\mu$ be partitions in a weight $2$ block of a $\q$-Schur algebra, with $\mu \geq \lambda$.
\begin{enumerate}
\item $\operatorname{Jan}_{\lambda\mu}(v) = d_{\lambda\mu}(v) = \rad_{\lambda\mu}(v)$.
\item $\Ext^1(L^\lambda,L^\mu) \ne 0$ if and only if $d_{\lambda\mu}(v) =v$, in which case $\Ext^1(L^\lambda,L^\mu)$ is one-dimensional.  In particular, $\Ext^1(L^\lambda,L^\mu) = 0$ unless $\sigma_e(\lambda) \ne \sigma_e(\mu)$.
\end{enumerate}
\end{thm}

\begin{proof}
Part(2) follows from the second equality of (1) since the dimension of $\Ext^1(L^\lambda,L^\mu)$ equals the composition multiplicity of $L^{\mu}$ in the second radical layer of $\Delta^\lambda$.

For part (1), note first that since the composition factors of the Weyl module $\Delta^\lambda$ are multiplicity-free by Corollary \ref{C:decomp}, so that for each $\mu$, $\Jan_{\lambda\mu}(v)$ is either a monic monomial or zero, according to whether $L^\mu$ is a composition factor of $\Delta^\lambda$ or not.  Furthermore, the following three statements are equivalent:
$$(1)\ \Jan_{\lambda\mu}(v) = 1; \qquad (2)\ \lambda =\mu; \qquad (3)\ d_{\lambda\mu}(v) = 1.$$
As such, to prove that $\Jan_{\lambda\mu}(v) = d_{\lambda\mu}(v)$, it suffices to show that $\frac{d}{dv} \Jan_{\lambda\mu}(v) |_{v=1} = \frac{d}{dv} d_{\lambda\mu}(v) |_{v=1}$, or equivalently, by Theorem \ref{T:Jantzen} and Theorem \ref{T:R-H}, that $J_{\lambda\mu}^{e,p} = J_{\lambda\mu}^{e,0}$.  But this holds by Corollary \ref{C:J}.

For $\rad_{\lambda\mu}(v)$, since the composition factors of $\Delta^\lambda$ are multiplicity-free by Corollary \ref{C:decomp}, and the Jantzen layers of $\Delta^\lambda$ are self-dual, we see that the Jantzen layers are in fact semi-simple, so that all the composition factors in the second Jantzen layer $\frac{\Delta^\lambda(1)}{\Delta^\lambda(2)}$ (which are precisely the composition factors of $\Delta^\lambda$ not of the form $L^\lambda$ or $L^{m(\lambda')}$) lie in the second radical layer $\frac{\rad(\Delta^\lambda)}{\rad^2(\Delta^\lambda)}$.  Thus, if $\lambda$ is not $e$-restricted (so that $m(\lambda')$ is not defined and $L^{m(\lambda')}$ does not exist), then $\Jan_{\lambda\mu}(v) = \rad_{\lambda\mu}(v)$ holds for all $\mu$.  If $\lambda$ is $e$-restricted, then since $\Delta^\lambda$ has a simple socle $L^{m(\lambda')}$ by Lemma \ref{L:basic}(2b), $\Jan_{\lambda\mu}(v) =  \rad_{\lambda\mu}(v)$ will hold for all $\mu$ if the second Jantzen layer is non-zero, or equivalently, if there exists some $\mu$ such that $d_{\lambda\mu}(v) = v$.  We now proceed to show this.

Note that there exists $\nu$ such that $d_{\nu m(\lambda')}(v) = v$: for, if not, then $d^{e,0}_{\nu m(\lambda')} = 0$ for all $\nu \notin \{ \lambda, m(\lambda')\}$ by Theorem \ref{T:wt2vdecomp} and Corollary \ref{C:decomp}, so that $J^{e,0}_{\lambda m(\lambda')} \ne 2$ (see Section \ref{S:Jantzen} for the definition of $J_{\lambda\mu}$), contradicting Theorems \ref{T:Mull}(1) and \ref{T:R-H}.  Thus $\Omega = \{ \nu \mid d_{\nu m(\lambda')}(v) = v\}$ is non-empty, and let $\mu$ be a minimal element (with respect to $\dom$) of $\Omega$.  Then $d_{\mu m(\lambda')}(v) = v$, so that $m(\lambda') \dom \mu \dom \lambda$ by Theorem \ref{T:wt2vdecomp}, and hence, by what we have already shown, $\Ext^1(L^{\mu}, L^{m(\lambda')})$ is one-dimensional.  The projective cover $P^{m(\lambda')}$ is self-dual by Lemma \ref{L:basic}(2a), and has a filtration with factors $\Delta^{m(\lambda')}$, $\Delta^{\nu}$ ($\nu \in \Omega$), and $\Delta^{\lambda}$.  Thus its heart $Q$ is self-dual, and has a filtration with factors $\rad(\Delta^{m(\lambda')})$, $\Delta^{\nu}$ ($\nu \in \Omega$) and $\Delta^{\lambda}/L^{m(\lambda')}$. Among these factors of $Q$, $L^{\mu}$ is a composition factor of $\Delta^{\mu}$ and possibly $\Delta^{\lambda}/L^{m(\lambda')}$ only, by the minimality of $\mu$.  Since $\Ext^1(L^{\mu}, L^{m(\lambda')})$ is one-dimensional, we have $L^\mu$ occurring exactly once in the head of $Q$.  Thus, this copy of $L^{\mu}$ must come from the head of the factor $\Delta^{\mu}$. As $d_{\mu m(\lambda')}$ ($= d_{\mu m(\lambda')}(v)|_{v=1}$ by Corollary \ref{C:decomp}) is non-zero, so that $\Delta^{\mu}$ is not simple, this copy of $L^{\mu}$ cannot lie in the socle of $Q$.  But by self-duality of $Q$, there must be a copy of $L^{\mu}$ in its socle, which can only come from $\Delta^{\lambda}/L^{m(\lambda')}$.  Thus $d_{\lambda\mu} \ne 0$, and hence $d_{\lambda\mu}(v)\ne 0$ by Corollary \ref{C:decomp}.  As $d_{\mu m(\lambda')}(v) = v$, this gives $\del \mu = \del m(\lambda') \pm 1 = \del \lambda \pm 1$ by Theorems \ref{T:wt2vdecomp} and \ref{T:R}(2c), so that $d_{\lambda\mu}(v) = v$.
\end{proof}

\begin{rem} \hfill
\begin{enumerate}
\item By applying the Schur functor to Theorem \ref{T:ext}, one can obtain the Ext-quivers and radical filtrations of Specht modules of the weight $2$ blocks of Iwahori-Hecke algebras where $p \ne 2$.  We omit the details here, but refer the interested reader to Section 6 of \cite{CT} in which the case where $\q=1$ and $p$ is an odd prime is dealt with.
\item The first equality of Theorem \ref{T:ext}(1) in particular proves that the conjecture (see \cite[Section 9]{LLT}) that the $v$-decomposition numbers describe the Jantzen filtration of the Weyl modules of $\q$-Schur algebras in characteristic $0$ holds for weight $2$ blocks.
\end{enumerate}
\end{rem}

\section{$[2:1]$-pairs}

In this section, we study the structure of $[2:1]$-pairs of $\q$-Schur algebras.  Throughout this section, we assume that $p \ne 2$, and $B$ and $C$ are weight 2 blocks of $\q$-Schur algebras forming a $[2:1]$-pair, where the $e$-core of $C$ can be obtained from that of $B$ by interchanging runners $(i-1)$ and $i$ of its abacus display.

There are exactly three partitions in $B$ which have more than one bead on runner $i$ of their respective abacus displays whose preceding positions on runner $(i-1)$ are vacant.  We call these partitions in $B$ {\em exceptional} and the others {\em non-exceptional}.  The exceptional ones are
labelled $\alpha$, $\beta$ and $\gamma$, and runners $(i-1)$ and $i$ of their respective abacus displays have the following form:
$$
\begin{matrix}
\alpha \\
\begin{smallmatrix}
\vdots & \vdots \\
\bullet & \bullet \\
\bullet & - \\
- & \bullet \\
- & \bullet \\
- & -
\end{smallmatrix}
\end{matrix} \qquad
\begin{matrix}
\beta \\
\begin{smallmatrix}
\vdots & \vdots \\
\bullet & \bullet \\
- & \bullet \\
\bullet & - \\
- & \bullet \\
- & -
\end{smallmatrix}
\end{matrix} \qquad
\begin{matrix}
\gamma \\
\begin{smallmatrix}
\vdots & \vdots \\
\bullet & \bullet \\
- & \bullet \\
- & \bullet \\
\bullet & - \\
- & -
\end{smallmatrix}
\end{matrix}
$$

Similarly, there are exactly three partitions in $C$ which have more than one bead on runner $(i-1)$ of their respective abacus displays whose succeeding positions on runner $i$ are vacant.  We call these partitions in $C$ {\em exceptional} and the others {\em non-exceptional}.  The exceptional ones are
labelled $\wt{\alpha}$, $\wt{\beta}$ and $\wt{\gamma}$, and runners $(i-1)$ and $i$ of their respective abacus displays have the following form:
$$
\begin{matrix}
\wt{\alpha} \\
\begin{smallmatrix}
\vdots & \vdots \\
\bullet & \bullet \\
\bullet & - \\
\bullet & - \\
- & \bullet \\
- & -
\end{smallmatrix}
\end{matrix} \qquad
\begin{matrix}
\wt{\beta} \\
\begin{smallmatrix}
\vdots & \vdots \\
\bullet & \bullet \\
\bullet & - \\
- & \bullet \\
\bullet & - \\
- & -
\end{smallmatrix}
\end{matrix} \qquad
\begin{matrix}
\wt{\gamma} \\
\begin{smallmatrix}
\vdots & \vdots \\
\bullet & \bullet \\
- & \bullet \\
\bullet & - \\
\bullet & - \\
- & -
\end{smallmatrix}
\end{matrix}
$$
Note that $\alpha$ is the unique partition in $B$ which has more than one normal bead on runner $i$, and $\wt{\alpha}$ is the unique partition in $C$ which has more than one conormal bead on runner $(i-1)$.

Let the conjugate block of $B$ be denoted by $B'$.  Thus $B'$ is the block having the same weight as $B$, and the $e$-core of $B'$ is the conjugate partition of the $e$-core of $B$.  Similarly, denote the conjugate block of $C$ by $C'$.  Then $B'$ and $C'$ also form a $[2:1]$-pair (with each other), with exceptional partitions are $\gamma'$, $\beta'$, $\alpha'$, and $\wt{\gamma}'$, $\wt{\beta}'$, $\wt{\alpha}'$ respectively.  

The following is well known for $[2:1]$-pairs:

\begin{thm} \label{T:21}
Let $\lambda$ be a partition in $B$ such that $\lambda \ne \alpha$.
\begin{enumerate}
\item $\Phi(\alpha) = \wt{\alpha}$, $\Phi(\beta) = \wt{\gamma}$, $\Phi(\gamma) = \wt{\beta}$.
\item $\del \lambda = \del \Phi(\lambda)$.
\item $\del\alpha = \del \gamma = \del \beta - 1$, $\del\wt{\alpha} = \del\wt{\gamma} = \del\wt{\beta} + 1$, $\del\alpha = \del \wt{\alpha} -1$.
\item
\begin{align*}
d_{\lambda\alpha} &=
\begin{cases}
1, &\text{if } \lambda \in \{\alpha, \beta, \gamma\}; \\
0, &\text{otherwise.}
\end{cases} \\
d_{\wt{\mu}\wt{\alpha}} &=
\begin{cases}
1, &\text{if } \wt{\mu} \in \{\wt{\alpha}, \wt{\beta}, \wt{\gamma}\}; \\
0, &\text{otherwise.}
\end{cases}
\end{align*}
In particular, $\gamma = m(\alpha)'$, $c_{\alpha\lambda} = c_{\lambda\alpha} = d_{\alpha\lambda} + d_{\beta\lambda}+d_{\gamma\lambda}$, and $\wt{\gamma} = m(\wt{\alpha})'$, $c_{\wt{\alpha}\wt{\mu}} = c_{\wt{\mu}\wt{\alpha}} = d_{\wt{\alpha}\wt{\mu}} + d_{\wt{\beta}\wt{\mu}} + d_{\wt{\gamma}\wt{\mu}}$.
\item \begin{alignat*}{2}
[\Delta^\alpha \down_{C}] &= [\Delta^{\wt{\alpha}}] + [\Delta^{\wt{\beta}}], \qquad
[\Delta^{\wt{\alpha}} \up^{B}] &= [\Delta^\alpha] + [\Delta^\beta]; \\
[\Delta^\beta \down_{C}] &= [\Delta^{\wt{\alpha}}] + [\Delta^{\wt{\gamma}}], \qquad
[\Delta^{\wt{\beta}} \up^{B}] &= [\Delta^\alpha] + [\Delta^\gamma]; \\
[\Delta^\gamma \down_{C}] &= [\Delta^{\wt{\beta}}] + [\Delta^{\wt{\gamma}}], \qquad
[\Delta^{\wt{\gamma}} \up^{B}] &= [\Delta^\beta] + [\Delta^\gamma].
\end{alignat*}
\item $c_{\alpha\lambda} \ne 0$ if and only if $c_{\wt{\alpha}\Phi(\lambda)} \ne 0$.
\item $[L^\alpha \down_{C} : L^{\wt{\alpha}}] = 2 = [L^{\wt{\alpha}} \up^{B} : L^\alpha]$.
\item $L^{\lambda} \down_{C} = L^{\Phi(\lambda)}$ and $L^{\Phi(\lambda)} \up^B = L^\lambda$.
\item If $\mu \notin \{\alpha,\beta,\gamma\}$, then $d_{\mu\lambda} = d_{\Phi(\mu)\Phi(\lambda)}$.
\end{enumerate}
\end{thm}

\begin{cor} \label{C:easy}
Let $\lambda$ be a partition in $B$ such that $\lambda \ne \alpha$.
Then
$$d_{\alpha\lambda} + d_{\wt{\gamma}\Phi(\lambda)} =
d_{\beta\lambda} + d_{\wt{\beta}\Phi(\lambda)} =
d_{\gamma\lambda} + d_{\wt{\alpha}\Phi(\lambda)}.$$
\end{cor}

\begin{proof}
By Theorem \ref{T:21}(4,5,8), we have
\begin{alignat*}{2}
d_{\wt{\alpha}\Phi(\lambda)} + d_{\wt{\beta}\Phi(\lambda)} &= [\Delta^\alpha \down_{C} : L^{\Phi(\lambda)}] &&= [L^\alpha \down_{C} : L^{\Phi(\lambda)}] + d_{\alpha\lambda}; \\
d_{\wt{\alpha}\Phi(\lambda)} + d_{\wt{\gamma}\Phi(\lambda)} &=[\Delta^\beta \down_{C} : L^{\Phi(\lambda)}] &&= [L^\alpha \down_{C} : L^{\Phi(\lambda)}] + d_{\beta\lambda}; \\
d_{\wt{\beta}\Phi(\lambda)} + d_{\wt{\gamma}\Phi(\lambda)} &=[\Delta^\gamma \down_{C} : L^{\Phi(\lambda)}] &&= [L^\alpha \down_{C} : L^{\Phi(\lambda)}] + d_{\gamma\lambda}.
\end{alignat*}
The Corollary thus follows.
\end{proof}

\begin{lem} \label{L:commonB}
Let $\lambda$ be a partition in $B$ such that $\lambda \ne \alpha$.  Then $d_{\alpha\lambda}$ and $d_{\beta\lambda}$ are both non-zero if and only if $\beta$ is $e$-restricted and $\lambda = m(\beta')$, and $d_{\beta\lambda}$ and $d_{\gamma\lambda}$ are both non-zero if and only if $\lambda = \beta$.
\end{lem}

\begin{proof}
If $d_{\alpha\lambda}, d_{\beta\lambda} \ne 0$, then either $\del\lambda = \del\alpha$, or $\del\lambda = \del\beta$ by Corollary \ref{C:decomp} and Theorem \ref{T:21}(3).  If $\del\lambda = \del\alpha$, then $\lambda = m(\alpha')$ (and $\alpha$ is $e$-restricted) by Corollary \ref{C:decomp}, since $d_{\alpha\lambda} \ne 0$ and $\lambda \ne \alpha$.  But $\beta \not\domeq \alpha = m(\lambda)'$, so that $d_{\beta\lambda} = 0$ by Lemma \ref{L:basic}(2), a contradiction.  Thus $\del\lambda = \del\beta$, so that $\lambda \in \{ \beta,m(\beta') \}$ by Corollary \ref{C:decomp}, since $d_{\beta\lambda} \ne 0$.  As $\beta \not\domeq \alpha$, we have $d_{\alpha\beta} = 0$ by Lemma \ref{L:basic}(1), and thus $\lambda \ne \beta$, giving $\lambda = m(\beta')$ (and $\beta$ is $e$-restricted).  Conversely, if $\beta$ is $e$-restricted and $\lambda = m(\beta')$, then $L^\lambda = L^{m(\beta')}$ is the socle of $\Delta^{\beta}$ by Lemma \ref{L:basic}(2b); in particular, $d_{\beta\lambda} \ne 0$.  By Theorems \ref{T:wt2vdecomp}, \ref{T:21}(3,4) and \ref{T:ext}(2), we see that $L^\alpha$ lies in the semi-simple heart of $\Delta^{\beta}$.  Thus, $\Ext^1(L^\alpha,L^\lambda) \ne 0$, so that either $d_{\alpha\lambda} \ne 0$ or $d_{\lambda\alpha} \ne 0$ by Theorem \ref{T:ext}(1) and Corollary \ref{C:decomp}.  As the latter cannot hold by Theorem \ref{T:21}(4), we have $d_{\alpha\lambda} \ne 0$.

If $d_{\beta\lambda}, d_{\gamma\lambda} \ne 0$, then either $\del\lambda = \del\beta$, or $\del\lambda = \del\gamma$ by Corollary \ref{C:decomp} and Theorem \ref{T:21}(3).  If $\del\lambda = \del\gamma$, then, by Corollary \ref{C:decomp}, $\lambda = \gamma$ since $d_{\gamma\lambda} \ne 0$ and $\lambda \ne \alpha = m(\gamma')$.  But $\gamma \not\domeq \beta$, so that $d_{\beta\lambda} = 0$ by Lemma \ref{L:basic}(1), a contradiction.  Thus $\del\lambda = \del\beta$, so that $\lambda \in \{ \beta,m(\beta') \}$ by Corollary \ref{C:decomp} (since $d_{\beta\lambda} \ne 0$).  If ($\beta$ is $e$-restricted and) $\lambda = m(\beta')$, then $m(\lambda)' = \beta \not\trianglelefteq \gamma$, so that $d_{\gamma\lambda} = 0$ by Lemma \ref{L:basic}(1), a contradiction.  Thus $\lambda = \beta$.  Conversely, if $\lambda = \beta$, then clearly $d_{\beta\lambda} \ne 0$.  Furthermore, one can easily verify by direct computation that $J_{\gamma\beta} = 1$, so that $d_{\gamma \beta} = 1$ by Theorem \ref{T:Jantzen}.
\end{proof}

\begin{lem} \label{L:commonwtB}
Let $\lambda$ be a partition in $B$ such that $\lambda \ne \alpha$.  Then $d_{\wt{\alpha}\Phi(\lambda)}$ and $d_{\wt{\beta}\Phi(\lambda)}$ are both non-zero if and only if $\wt{\beta}$ is $e$-restricted and $\Phi(\lambda) = m(\wt{\beta}')$, and $d_{\wt{\beta}\Phi(\lambda)}$ and $d_{\wt{\gamma}\Phi(\lambda)}$ are both non-zero if and only if $\Phi(\lambda) = \wt{\beta}$.
\end{lem}

\begin{proof}
An argument entirely analogous to that used in Lemma \ref{L:commonB} applies.
\end{proof}

\begin{cor} \label{C:c}
Let $\lambda$ be a partition in $B$ such that $\lambda \ne \alpha$.
\begin{enumerate}
\item Both $c_{\alpha\lambda}$ and $c_{\wt{\alpha}\Phi(\lambda)}$ are bounded above by $2$.
\item If $c_{\alpha\lambda} \ne 0$, then
$d_{\alpha\lambda} + d_{\wt{\gamma}\Phi(\lambda)} =
d_{\beta\lambda} + d_{\wt{\beta}\Phi(\lambda)} =
d_{\gamma\lambda} + d_{\wt{\alpha}\Phi(\lambda)} = 1$; in particular, $c_{\alpha\lambda} + c_{\wt{\alpha}\Phi(\lambda)} = 3$.
\end{enumerate}
\end{cor}

\begin{proof}
(1) follows directly from Theorem \ref{T:21}(4) and Lemmas \ref{L:commonB} and \ref{L:commonwtB}.  For (2), let $d_{\alpha\lambda} + d_{\wt{\gamma}\Phi(\lambda)} =
d_{\beta\lambda} + d_{\wt{\beta}\Phi(\lambda)} =
d_{\gamma\lambda} + d_{\wt{\alpha}\Phi(\lambda)} = r$ (the first two equalities follow from Corollary \ref{C:easy}). Summing up the three equations, we get $c_{\alpha\lambda} + c_{\wt{\alpha}\Phi(\lambda)} = 3r$ by Theorem \ref{T:21}(4).  Since $0< c_{\alpha\lambda} + c_{\wt{\alpha}\Phi(\lambda)} \leq 4$ by (1), and $r$ is an integer, we must have $r=1$, and hence $c_{\alpha\lambda} + c_{\wt{\alpha}\Phi(\lambda)} = 3$.
\end{proof}

\begin{prop} \label{P:prelim}
Let $\lambda$ be a partition in $B$ such that $\lambda \ne \alpha$.
\begin{enumerate}
\item If $c_{\alpha\lambda} = 2$, then $[L^{\wt{\alpha}} \up^B : L^\lambda] = 1$.
\item If $c_{\alpha\lambda} = 1$, then $[L^{\wt{\alpha}} \up^B : L^\lambda] = 0$.
\item If $c_{\wt{\alpha}\Phi(\lambda)} = 2$, then $[L^{\alpha} \down_{C} : L^{\Phi(\lambda)}] = 1$.
\item If $c_{\wt{\alpha}\Phi(\lambda)} = 1$, then $[L^{\alpha} \down_{C} : L^{\Phi(\lambda)}] = 0$.
\end{enumerate}
\end{prop}

\begin{proof}
(1) follows from Theorem \ref{T:21}(4,5,8) and Corollary \ref{C:decomp}.  For example, if $d_{\alpha\lambda} = d_{\beta\lambda} = 1$ and $d_{\gamma\lambda} = 0$, then $1 \leq [L^{\wt{\alpha}} \up^B : L^\lambda] \leq 2$ follows from the equation $[\Delta^{\wt{\alpha}} \up^{B}] = [\Delta^\alpha] + [\Delta^\beta]$ together with Theorem \ref{T:21}(4,8) and the fact that $d_{\wt{\alpha}\Phi(\lambda)} \leq 1$ (by Corollary \ref{C:decomp}), while $[L^{\wt{\alpha}} \up^B : L^\lambda] \leq 1$ follows from the equation $[\Delta^{\wt{\beta}} \up^{B}] = [\Delta^\alpha] + [\Delta^\gamma]$ together with Theorem \ref{T:21}(4).  Thus $[L^{\wt{\alpha}} \up^B : L^\lambda] = 1$

(2) follows from Theorem \ref{T:21}(4,5).  For example, when $d_{\alpha\lambda} = 1$ and $d_{\beta\lambda} = d_{\gamma\lambda} = 0$, then $[L^{\wt{\alpha}} \up^B : L^\lambda] = 0$ follows from $[\Delta^{\wt{\gamma}} \up^B] = [\Delta^{\beta}] + [\Delta^{\gamma}]$.

(3) and (4) follow using arguments analogous to those for (1) and (2) respectively.
\end{proof}

\begin{cor} \label{C:ll3}
$L^{\alpha} \down_{C}$ has a simple head and a simple socle both isomorphic to $L^{\wt{\alpha}}$, and a non-zero heart which is multiplicity-free; in particular, the radical length of $L^{\alpha} \down_{C}$ is $3$. Similarly, $L^{\wt{\alpha}} \up^B$ has a simple head and a simple socle both isomorphic to $L^{\alpha}$, and a non-zero heart which is multiplicity-free; in particular, the radical length of $L^{\wt{\alpha}} \up^{B}$ is $3$.
\end{cor}

\begin{proof}
Since $L^{\alpha}$, and hence $L^{\alpha} \down_{C}$, is self-dual, we see that $L^{\alpha} \down_{C}$ has a simple head and a simple socle both isomorphic to $L^{\wt{\alpha}}$ by Theorems \ref{T:Bru} and \ref{T:21}(1).  Now, the heart $L^{\alpha} \down_{C}$ is self-dual (by the self-duality of $L^{\alpha} \down_{C}$), multiplicity-free (by Theorem \ref{T:21}(7), Proposition \ref{P:prelim} and Corollary \ref{C:c}) and non-zero (otherwise $L^{\wt{\alpha}}$ will self-extend, contradicting Theorem \ref{T:ext}(1)).  Hence it is semi-simple and has radical length $1$.  Thus $L^{\alpha} \down_{C}$ has radical length $3$.  This proves the first assertion.

The second assertion follows from an entirely analogous argument.
\end{proof}

\begin{prop} \label{P:equiv2}
Let $\lambda$ be a partition in $B$ such that $\lambda \ne \alpha$.  The following statements are equivalent:
\begin{enumerate}
\item $c_{\alpha\lambda} = 2$.
\item $[L^{\wt{\alpha}} \up^B : L^\lambda] \ne 0$.
\item $\Ext^1(L^\alpha,L^\lambda) \ne 0$.
\item $\Ext^1(L^{\wt{\alpha}},L^{\Phi(\lambda)}) = 0$ and $c_{\wt{\alpha}\Phi(\lambda)} \ne 0$.
\item $[L^\alpha \down_{C} : L^{\Phi(\lambda)}] = 0$ and $c_{\wt{\alpha}\Phi(\lambda)} \ne 0$.
\item $c_{\wt{\alpha}\Phi(\lambda)} = 1$.
\end{enumerate}
\end{prop}

\begin{proof}
By Proposition \ref{P:prelim}, we see that (1) implies (2), and (5) implies (6).  By Corollary \ref{C:ll3}, we see that (2) implies (3), and (4) implies (5).  If (3) holds, then clearly $c_{\alpha\lambda} \ne 0$ so that $c_{\wt{\alpha}\Phi(\lambda)} \ne 0$ by Theorem \ref{T:21}(6); furthermore, $\sigma_e(\lambda) \ne \sigma_e(\alpha)$ by Theorem \ref{T:ext}, so that $\sigma_e(\Phi(\lambda)) = \sigma_e(\wt{\alpha})$ by Theorem \ref{T:21}(2,3) and hence $\Ext^1(L^{\wt{\alpha}},L^{\Phi(\lambda)}) = 0$ by Theorem \ref{T:ext}.  Finally, that (6) implies (1) follows from Corollary \ref{C:c}(2).
\end{proof}

Similarly, we also have the analogue of Proposition \ref{P:equiv2}.

\begin{prop} \label{P:equiv1}
Let $\lambda$ be a partition in $B$ such that $\lambda \ne \alpha$.  The following statements are equivalent:
\begin{enumerate}
\item $c_{\alpha\lambda} = 1$.
\item $[L^{\wt{\alpha}} \up^B : L^\lambda] = 0$ and $c_{\alpha\lambda} \ne 0$.
\item $\Ext^1(L^\alpha,L^\lambda) = 0$ and $c_{\alpha\lambda} \ne 0$.
\item $\Ext^1(L^{\wt{\alpha}},L^{\Phi(\lambda)}) \ne 0$.
\item $[L^\alpha \down_{C} : L^{\Phi(\lambda)}] \ne 0$.
\item $c_{\wt{\alpha}\Phi(\lambda)} = 2$.
\end{enumerate}
\end{prop}

\section{Alvis-Curtis duality}

In view of Corollary \ref{C:decomp}, it is in principle possible to compute $e_{\lambda\mu}$ from our knowledge of $d_{\lambda\mu}$, and then use $d_{\lambda\mu}$ and $e_{\lambda\mu}$ to compute $a_{\lambda\mu}$.  However, our attempts to find a nice closed formula for $e_{\lambda\mu}$ fail, and we compute $a_{\lambda\mu}$ by a more roundabout way.

We continue to assume that $p\ne 2$.

\begin{prop} \label{P:main}
Suppose $B$ and $C$ are weight 2 blocks of $\q$-Schur algebras forming a $[2:k]$-pair.  Denote the conjugate blocks of $B$ and $C$ by $B'$ and $C'$ respectively.  Let $\lambda$ and $\mu$ be partitions in $B$ and $B'$ respectively, and write $\Phi = \Phi_{B,C}$ and $\Phi' = \Phi_{B',C'}$. Then
$$
a_{\lambda\mu} = a_{\Phi(\lambda) \Phi'(\mu)},
$$
unless $k = 1$, $\lambda = \alpha$ and $\mu \ne \gamma'$.
\end{prop}

First we note the following fact which we shall use: if $\nu$ is a partition in $B$, then, unless $k=1$ and $\nu \in \{ \alpha,\beta,\gamma\}$, the effect of $\Phi$ on $\nu$ as well as the effect of $\Phi'$ on $\nu'$ is merely to interchange two runners, so that $\Phi(\nu)' = \Phi'(\nu')$.

When $k\geq 2$, we have
\begin{alignat*}{2}
L^{\lambda} \down_C &= L^{\Phi(\lambda)}, \quad& L^{\Phi(\lambda)} \up^B &= L^\lambda, \\
[\Delta^{\lambda} \down_C] &= [\Delta^{\Phi(\lambda)}], \quad& [\Delta^{\Phi(\lambda)} \up^B] &= [\Delta^\lambda],
\end{alignat*}
for all partitions $\lambda$ in $B$.  This implies $d_{\lambda\rho} = d_{\Phi(\lambda)\Phi(\rho)}$ for all partitions $\lambda$ and $\rho$ in $B$, and hence $e_{\lambda\rho} = e_{\Phi(\lambda)\Phi(\rho)}$ for all partitions $\lambda$ and $\rho$ in $B$.  Similarly, $d_{\tau\mu} = d_{\Phi'(\tau)\Phi'(\mu)}$ for all partitions $\tau$ and $\mu$ of $B'$.  Proposition \ref{P:main} thus follows, since $a_{\lambda\mu} = \sum_{\nu} e_{\lambda\nu}d_{\nu'\mu}$, and $\Phi(\rho)' = \Phi'(\rho')$ for all partitions $\rho$ in $B$.

As such, to prove Proposition \ref{P:main}, it suffices to consider the case where $k=1$.  Thus, let $B$ and $C$ be two weight $2$ blocks of $\q$-Schur algebras forming a $[2:1]$-pair, and we keep all the notations introduced in the last section.

We begin with an easy Lemma.

\begin{lem} \label{L:ez}
Suppose $B$ and $C$ form a $[2:1]$-pair, and let $\lambda$ be a partition in $B$ and let $\wt{\mu}$ be a partition in $C$.  Then $e_{\lambda \alpha} + e_{\lambda \beta} + e_{\lambda\gamma} = \delta_{\lambda\alpha}$ and $e_{\wt{\mu}\wt{\alpha}} + e_{\wt{\mu}\wt{\beta}} + e_{\wt{\mu}\wt{\gamma}} = \delta_{\wt{\mu}\wt{\alpha}}$.
\end{lem}

\begin{proof}
We have $\sum_{\nu} e_{\lambda\nu}\,d_{\nu\alpha} = \delta_{\lambda\alpha}$ and $\sum_{\wt{\nu}} e_{\wt{\mu}\wt{\nu}}\,d_{\wt{\nu}\wt{\alpha}} = \delta_{\wt{\mu}\wt{\alpha}}$, so that this follows from Theorem \ref{T:21}(4).
\end{proof}

\begin{prop} \label{P:maine}
Suppose $B$ and $C$ form a $[2:1]$-pair, and let $\lambda$ and $\mu$ be partitions in $B$.
\begin{enumerate}
\item If $\lambda \ne \alpha$, and $\mu \notin \{ \alpha, \beta, \gamma \}$, then $e_{\lambda\mu} = e_{\Phi(\lambda)\Phi(\mu)}$.

\item If $\lambda \ne \alpha$, then
\begin{alignat*}{2}
e_{\lambda\alpha} &= e_{\Phi(\lambda)\wt{\alpha}} + e_{\Phi(\lambda)\wt{\beta}}, & \quad e_{\Phi(\lambda)\wt{\alpha}} &= e_{\lambda\alpha} + e_{\lambda\beta}; \\
e_{\lambda\beta} &= e_{\Phi(\lambda)\wt{\alpha}} + e_{\Phi(\lambda)\wt{\gamma}}, & \quad e_{\Phi(\lambda)\wt{\beta}} &= e_{\lambda\alpha} + e_{\lambda\gamma}; \\
e_{\lambda\gamma} &= e_{\Phi(\lambda)\wt{\beta}} + e_{\Phi(\lambda)\wt{\gamma}}, & \quad e_{\Phi(\lambda)\wt{\gamma}} &= e_{\lambda\beta} + e_{\lambda\gamma}.
\end{alignat*}

\item If $\mu \notin \{ \alpha, \beta, \gamma \}$, then
\begin{align*}
e_{\alpha\mu} &= \tfrac{1}{2} \left( e_{\wt{\alpha}\Phi(\mu)} - \sum_{\wt{\nu}} e_{\wt{\nu}\Phi(\mu)} \right); \\
e_{\wt{\alpha}\Phi(\mu)} &= \tfrac{1}{2} \left( e_{\alpha\mu} - \sum_{\nu} e_{\nu\mu} \right),
\end{align*}
where $\wt{\nu}$ and $\nu$ runs over all partitions satisfying $c_{\wt{\alpha}\wt{\nu}} = 1$ and $c_{\alpha\nu} = 1$ respectively.

\item $e_{\beta\alpha} = e_{\wt{\beta}\wt{\alpha}} = -1$, $e_{\gamma\alpha} = e_{\wt{\gamma}\wt{\alpha}} = 0$, $e_{\gamma\beta} = e_{\wt{\gamma}\wt{\beta}} = -1$.

\end{enumerate}
\end{prop}

\begin{proof}
We have $[L^\lambda] = \sum_{\rho} e_{\lambda \rho} [\Delta^\rho]$.  We restrict both sides of this equation to the block $C$.  On the right-hand side, we have
$$ (e_{\lambda \alpha} + e_{\lambda \beta})[\Delta^{\wt{\alpha}}] +
(e_{\lambda \alpha} + e_{\lambda \gamma})[\Delta^{\wt{\beta}}] +
(e_{\lambda \beta} + e_{\lambda \gamma})[\Delta^{\wt{\gamma}}] +
\sum_{\rho \notin \{\alpha,\beta, \gamma\}} e_{\lambda\rho}[\Delta^{\Phi(\rho)}].
$$
by Theorem \ref{T:21}(5).

If $\lambda \ne \alpha$, then on the left-hand side, we have $[L^{\lambda} \down_{C}] = [L^{\Phi(\lambda)}] = \sum_{\rho} e_{\Phi(\lambda)\Phi(\rho)} [\Delta^{\Phi(\rho)}]$ by Theorem \ref{T:21}(8).  Equating the coefficients of $[\Delta^{\Phi(\mu)}]$ on both sides gives (1) and the three equations on the right in (2).  The three equations on the left in (2) then follow from Lemma \ref{L:ez}.

If $\lambda = \alpha$, then on the left-hand side, we have
\begin{align*}
[L^{\alpha} \down_C] &= 2[L^{\wt{\alpha}}] + \sum_{\nu} [L^{\Phi(\nu)}] \\
&= \sum_{\rho} (2e_{\wt{\alpha}\Phi(\rho)} + \sum_{\nu} e_{\Phi(\nu)\Phi(\rho)})[\Delta^{\Phi(\rho)}],
\end{align*}
where $\nu$ runs over all partitions in $B$ satisfying $\nu \ne \alpha$ and $[L^\alpha \down_{C}: L^{\Phi(\nu)}] \ne 0$, by Theorem \ref{T:21}(7) and Corollary \ref{C:ll3}.  Equating the coefficients of $[\Delta^{\Phi(\mu)}]$ ($\mu \notin \{ \alpha, \beta, \gamma \}$) on both sides, we get
$$
  2e_{\wt{\alpha}\Phi(\mu)} + \sum_{\nu} e_{\Phi(\nu)\Phi(\mu)} = e_{\alpha\mu}.
$$
Using (1) and Proposition \ref{P:equiv1}, this yields the second assertion of (3).  The first assertion of (3) follows from an entirely analogous argument.

The first assertion of (4) follows from Lemma \ref{L:ez} and Corollary \ref{C:basice}.  The second assertion then follows from the first assertion and part (2) (and Theorem \ref{T:21}(1)). The third assertion now follows from the second and Lemma \ref{L:ez}.
\end{proof}

\begin{proof}[Proof of Proposition \ref{P:main}]
As mentioned earlier, we may assume $k=1$.  If $\mu = \gamma'$, then $a_{\lambda\mu} = \delta_{\lambda m(\mu)} = \delta_{\lambda \alpha}$ and $a_{\Phi(\lambda)\Phi'(\mu)} = \delta_{\Phi(\lambda) m(\Phi'(\mu))} = \delta_{\Phi(\lambda)\wt{\alpha}}$ by Theorems \ref{Mullineux} and \ref{T:Bru}(1), so that $a_{\lambda\mu} = a_{\Phi(\lambda) \Phi'(\mu)}$.

Thus, we may further assume that $\mu \ne \gamma'$, and hence that $\lambda \ne \alpha$ as well. We have
{\allowdisplaybreaks
\begin{align*}
a_{\lambda\mu} &= \sum_{\nu} e_{\lambda\nu}\,d_{\nu'\mu} \\
&= e_{\lambda\alpha}\,d_{\alpha'\mu} + e_{\lambda\beta}\,d_{\beta'\mu} + e_{\lambda\gamma}\,d_{\gamma'\mu} + \sum_{\nu \notin \{\alpha,\beta,\gamma\}} e_{\lambda\nu}\,d_{\nu'\mu} \\
&= e_{\lambda\alpha}\,d_{\alpha'\mu} + e_{\lambda\beta}\,d_{\beta'\mu} + e_{\lambda\gamma}\,d_{\gamma'\mu} + \sum_{\wt{\nu} \notin \{\wt{\alpha},\wt{\beta},\wt{\gamma}\}} e_{\Phi(\lambda)\wt{\nu}}\,d_{\wt{\nu}'\Phi'(\mu)} \\
&= a_{\Phi(\lambda)\Phi'(\mu)} + e_{\lambda\alpha}\,d_{\alpha'\mu} + e_{\lambda\beta}\,d_{\beta'\mu} + e_{\lambda\gamma}\,d_{\gamma'\mu} - \\
& \qquad
(e_{\Phi(\lambda)\wt{\alpha}}\,d_{\wt{\alpha}'\Phi'(\mu)} + e_{\Phi(\lambda)\wt{\beta}}\,d_{\wt{\beta}'\Phi'(\mu)} + e_{\Phi(\lambda)\wt{\gamma}}\,d_{\wt{\gamma}'\Phi'(\mu)}) \\
&= a_{\Phi(\lambda)\Phi'(\mu)} + e_{\lambda\alpha} (d_{\alpha'\mu} - d_{\wt{\alpha}'\Phi'(\mu)} - d_{\wt{\beta}'\Phi'(\mu)} ) +\\
& \qquad
e_{\lambda\beta} (d_{\beta'\mu} - d_{\wt{\alpha}'\Phi'(\mu)} - d_{\wt{\gamma}'\Phi'(\mu)}) + e_{\lambda\gamma} (d_{\gamma'\mu} - d_{\wt{\beta}'\Phi(\mu)} - d_{\wt{\gamma}'\Phi(\mu)}) \\
&= a_{\Phi(\lambda)\Phi'(\mu)} + e_{\lambda\alpha} (d_{\alpha'\mu}+d_{\wt{\gamma'}\Phi'(\mu)} - c_{\wt{\gamma}'\Phi'(\mu)}) +\\
& \qquad
e_{\lambda\beta} (d_{\beta'\mu}+d_{\wt{\beta}'\Phi'(\mu)} - c_{\wt{\gamma}'\Phi'(\mu)}) + e_{\lambda\gamma} (d_{\gamma'\mu}+d_{\wt{\alpha}'\Phi'(\mu)} - c_{\wt{\gamma}'\Phi'(\mu)}) \\
&= a_{\Phi(\lambda)\Phi'(\mu)} + (e_{\lambda\alpha} + e_{\lambda\beta} + e_{\lambda\gamma})(r - c_{\wt{\gamma}'\Phi'(\mu)}) \\
&= a_{\Phi(\lambda)\Phi'(\mu)}.
\end{align*}
}
Here, the third equality follows from Proposition \ref{P:maine}(1) and Theorem \ref{T:21}(9) (and the fact that $\Phi(\nu)' = \Phi'(\nu')$ for all partitions $\nu \notin \{\alpha,\beta,\gamma\}$ in $B$), the fifth from Proposition \ref{P:maine}(2) and the sixth from Theorem \ref{T:21}(4); the penultimate equality follows since $d_{\alpha'\mu} + d_{\wt{\gamma}'\Phi'(\mu)} = d_{\beta'\mu} + d_{\wt{\beta}'\Phi'(\mu)} =d_{\gamma'\mu} + d_{\wt{\alpha}'\Phi'(\mu)}$ by Corollary \ref{C:easy}, and we write this common quantity as $r$, while the final equality follows from Lemma \ref{L:ez}.
\end{proof}

We are left to address how $a_{\alpha\mu}$ with $\mu \ne \gamma'$ changes through a $[2:1]$-pair.

\begin{lem} \label{L:simple}
If $c_{\alpha\nu} = 1$ and $\nu \ne \gamma$, then $e_{\nu\alpha} = e_{\nu \beta} = e_{\nu \gamma} = 0$.  If $c_{\wt{\alpha}\wt{\nu}} = 1$ and $\wt{\nu} \ne \wt{\gamma}$, then $e_{\wt{\nu} \wt{\alpha}} = e_{\wt{\nu}\wt{\beta}} = e_{\wt{\nu}\wt{\gamma}} = 0$.
\end{lem}

\begin{proof}
If $c_{\alpha\nu} = 1 = d_{\gamma\nu}$, then $d_{\wt{\alpha}\Phi(\nu)} = 0$ and $d_{\wt{\beta}\Phi(\nu)} = d_{\wt{\gamma}\Phi(\nu)} = 1$ by Corollary \ref{C:c}(2), so that $\Phi(\nu) = \wt{\beta}$ by Lemma \ref{L:commonwtB}, and hence $\nu = \gamma$ by Theorem \ref{T:21}(1).  Thus, if $c_{\alpha\nu} = 1$ and $\nu \ne \gamma$, then either $d_{\alpha\nu} = 1$ or $d_{\beta\nu} = 1$.  In both cases, we have $\nu \dom \beta$.  Thus $e_{\nu\beta} = 0 = e_{\nu\gamma}$ by Corollary \ref{C:basice}.  That $e_{\nu\alpha} = 0$ now follows from Lemma \ref{L:ez}.  An analogous argument applies to the second assertion.
\end{proof}

\begin{prop} \label{P:alpha}
Suppose $B$ and $C$ form a $[2:1]$-pair, and let $\mu$ be a partition in $B'$ such that $\mu \ne \gamma'$.  Then
\begin{align*}
a_{\alpha\mu} &= \lceil \tfrac{1}{2} (a_{\wt{\alpha}\Phi'(\mu)} - \sum_{\wt{\nu}} a_{\wt{\nu}\Phi'(\mu)} ) \rceil; \\
a_{\wt{\alpha}\Phi'(\mu)} &= \lceil \tfrac{1}{2} (a_{\alpha\mu} - \sum_{\nu} a_{\nu\mu} ) \rceil,
\end{align*}
where $\wt{\nu}$ and $\nu$ run over all partitions satisfying $c_{\wt{\alpha}\wt{\nu}} = 1$ and $c_{\alpha\nu} = 1$ respectively.
\end{prop}

\begin{proof}
We have
\begin{align*}
a_{\alpha\mu}&= d_{\alpha'\mu} + \sum_{\rho \notin \{\alpha,\beta,\gamma\}} e_{\alpha\rho}d_{\rho'\mu} \\
&= d_{\alpha'\mu} + \sum_{\wt{\rho} \notin \{\wt{\alpha},\wt{\beta},\wt{\gamma} \}} \tfrac{1}{2} \left( e_{\wt{\alpha}\wt{\rho}} - \sum_{\wt{\nu}} e_{\wt{\nu}\wt{\rho}} \right) d_{\wt{\rho}'\Phi'(\mu)} \\
&= d_{\alpha'\mu} + \tfrac{1}{2}(-d_{\wt{\alpha}'\Phi'(\mu)} - d_{\wt{\beta}'\Phi'(\mu)} + d_{\wt{\gamma}'\Phi'(\mu)} + a_{\wt{\alpha}\Phi'(\mu)} - \sum_{\wt{\nu}} a_{\wt{\nu}\Phi'(\mu)} ),
\end{align*}
where $\wt{\nu}$ runs over all partitions satisfying $c_{\wt{\alpha}\wt{\nu}} = 1$. Here, the first equality follows from Corollary \ref{C:basice}, the second from Proposition \ref{P:maine}(3) and Theorem \ref{T:21}(9) (and that $\Phi(\rho)' = \Phi'(\rho')$ for all partitions $\rho \notin \{\alpha,\beta,\gamma\}$ in $B$), and the last from Corollary \ref{C:basice}, Proposition \ref{P:maine}(4) and Lemma \ref{L:simple}. If $c_{\gamma'\mu}= 0$, then $a_{\alpha\mu} = \frac{1}{2}(a_{\wt{\alpha}\Phi'(\mu)} - \sum_{\wt{\nu}} a_{\wt{\nu}\Phi'(\mu)})$ by Theorem \ref{T:21}(4,6).  If $c_{\gamma'\mu} \ne 0$, then $d_{\alpha'\mu} = 1 - d_{\wt{\gamma}'\Phi'(\mu)}$ by Corollary \ref{C:c}(2), so that
$$
a_{\alpha\mu} =
\begin{cases}
\frac{1}{2}(a_{\wt{\alpha}\Phi'(\mu)} - \sum_{\wt{\nu}} a_{\wt{\nu}\Phi'(\mu)}),
&\text{if } c_{\wt{\gamma}'\Phi'(\mu)} = 2; \\[3pt]
\frac{1}{2}(1+a_{\wt{\alpha}\Phi'(\mu)} - \sum_{\wt{\nu}} a_{\wt{\nu}\Phi'(\mu)}),
&\text{if } c_{\wt{\gamma}'\Phi'(\mu)} = 1.
\end{cases}
$$
Since $a_{\alpha\mu}$ is necessarily an integer, the proof of the first assertion is complete.  The second assertion follows from an entirely analogous argument.
\end{proof}

\begin{cor}[of proof]
Suppose $B$ and $C$ form a $[2:1]$-pair, and let $\mu$ be a partition in $B'$ such that $\mu \ne \gamma'$.  Then $a_{\alpha\mu} - \sum_{\nu} a_{\nu\mu}$, where $\nu$ runs over all partitions such that $c_{\alpha\nu} = 1$, is odd if and only if $c_{\gamma'\mu} = 1$, and $a_{\wt{\alpha}\Phi(\mu)} - \sum_{\wt{\nu}} a_{\wt{\nu}\Phi(\mu)}$, where $\wt{\nu}$ runs over all partitions such that $c_{\wt{\alpha}\wt{\nu}} = 1$, is odd if and only if $c_{\wt{\gamma}'\Phi(\mu)} = 1$.
\end{cor}

We now introduce a new labelling of partitions having $e$-weight $2$, due to Chuang and Turner \cite{CTu}.

\begin{Def}
Let $\lambda$ be a partition having $e$-weight $2$.
\begin{itemize}
\item If the abacus display of $\lambda$ has a bead at position $x$ and a bead at position $x-e$, while position $x-2e$ is vacant, and there are exactly $a$ vacant positions between $x$ and $x-e$, and $b$ vacant positions between $x-e$ and $x-2e$, then $\lambda = [a,b]$.  Note that $\del\lambda = a-b$.

\item If the abacus display of $\lambda$ has a bead at position $x$ which two vacant positions above it, i.e.\ positions $x-e$ and $x-2e$ are vacant, and there are exactly $a$ vacant positions between $x$ and $x-e$, and $b$ vacant positions between $x-e$ and $x-2e$ (inclusive of $x-e$), then $\lambda = [a,b]$.  Note that $\del\lambda = a-b+1$.

\item If the abacus display of $\lambda$ has two beads, at positions $x$ and $y$ say, with $x > y$, $x \not\equiv y \pmod e$, each with a vacant position above it, i.e.\ positions $x-e$ and $y-e$ are vacant, and there are exactly $a$ vacant positions between $x$ and $x-e$, and $b$ vacant positions between $y$ and $y-e$, then $\lambda = [a,b]$.  Note that
$$
\del\lambda =
\begin{cases}
a-b+1, &\text{if }x-e < y < x; \\
a-b, &\text{otherwise.}
\end{cases}
$$
\end{itemize}
\end{Def}

When we need to emphasize $[a,b]$ is a partition in the weight $2$ block $B$, we write it as $[a,b]_B$.  Clearly, this labelling of $\lambda$ is independent of the abacus used to display $\lambda$.

\begin{eg*}
Let $B$ be the canonical weight $2$ Rouquier block, and we display the partitions in $B$ on an abacus in which runner $i$ has $(i+2)$ beads for all $i$.  For $0 \leq b < a < e$, the partition $[a,b]$ is obtained from its $e$-core by sliding one bead on each of the runners $a$ and $b$ down one position.  For $0 \leq a < e$, the partition $[a,a]$ is obtained by sliding two beads on runner $a$ down one position each, and $[a,a+1]$ is obtained by sliding the bottom bead on runner $a$ down two positions.
\end{eg*}

The main advantage of this labelling is that over a $[2:k]$-pair, it is invariant under the action of $\Phi$.  More precisely,

\begin{lem}[{\cite[Proposition 95]{CTu}}] \label{L:invar}
Suppose $B$ and $C$ form a $[2:k]$-pair.  Let $[a,b]_B$ be a partition in $B$.  Then $\Phi([a,b]_B) = [a,b]_C$.
\end{lem}

\begin{proof}
Let $\lambda = [a,b]_B$.  If the abacus display of $\lambda$ has exactly one bead on runner $i$ whose preceding position is vacant, then the effect of $\Phi$ on $\lambda$ is to interchange runners $i$ and $(i-1)$.  In this case, it is easy to see that $\Phi(\lambda) = \Phi([a,b]_B) = [a,b]_C$.  On the other hand, if the abacus display of $\lambda$ has more than one bead on runner $i$ whose preceding position is vacant, then $k=1$ and $\lambda$ is an exceptional partition, i.e.\ $\lambda \in \{\alpha,\beta,\gamma\}$, and one can also verify in this case that $\Phi(\lambda) = \Phi([a,b]_B) = [a,b]_C$ using Theorem \ref{T:21}(1).
\end{proof}

The Mullineux map on weight 2 partitions can thus be easily described under this labelling:

\begin{lem} \label{L:Mull}
Let $B$ be a weight $2$ block, and let $B'$ be its conjugate block.  Then
\begin{alignat*}{2}
m([a,a+1]_B) &= [e-a,e-a+1]_{B'} &\qquad &(1 \leq a < e); \\
m([a,b]_B) &= [e-b,e-a]_{B'} &&(1 \leq b \leq a < e).
\end{alignat*}
\end{lem}

\begin{proof}
Using Proposition 3.7(2) of \cite{T1}, one can check that this holds for the canonical weight 2 Rouquier block.  If $B$ and $C$ form a $[2:k]$-pair, and $C'$ is the conjugate block of $C$, then $m(\Phi_{B,C}(\lambda)) = \Phi_{B',C'}(m(\lambda))$, so that the Lemma holds for $B$ if and only if it holds for $C$.  Since every arbitrary weight $2$ block can be induced to a Rouquier block by Lemma \ref{L:sequence}, and the Rouquier blocks of a given weight form a single Scopes equivalence class, the Lemma follows.
\end{proof}

\begin{Def}
Let $\lambda = [a,b]$ be a partition having $e$-weight $2$. Define
$$
\varepsilon \lambda =
\begin{cases}
1, &\text{if } \del\lambda \ne a-b; \\
0, &\text{otherwise.}
\end{cases}
$$
\end{Def}

\begin{eg*}
Consider the partitions in the canonical weight $2$ Rouquier block.  Then
$$\varepsilon \lambda =
\begin{cases}
1, &\text{if } \lambda \in \{ [a,a+1] \mid 0 \leq a < e \}; \\
0, &\text{otherwise.}
\end{cases}
$$
\end{eg*}

\begin{lem} \label{L:ep}
Suppose $B$ and $C$ form a $[2:k]$-pair, and let $\lambda$ be a partition in $B$. Then $\varepsilon \lambda = \varepsilon \Phi(\lambda)$ unless $k=1$ and $\lambda = \alpha$, in which case $\varepsilon \alpha = 0$ and $\varepsilon \wt{\alpha} = 1$.
\end{lem}

\begin{proof}
Note that $\del\lambda = \del\Phi(\lambda)$ unless $k=1$ and $\lambda = \alpha$, in which case $\del\lambda = \del\Phi(\lambda) - 1$ (cf.\ Theorem \ref{T:21}(1,2,3)).  By Lemma \ref{L:invar}, if $\lambda = [a,b]_B$, then $\Phi(\lambda) = [a,b]_C$. Thus, unless $k=1$ and $\lambda = \alpha$, we have $\varepsilon \lambda = \varepsilon \Phi(\lambda)$.  That $\varepsilon \alpha = 0$ and $\varepsilon \wt{\alpha} = 1$ follows directly from the definitions.
\end{proof}

We are now able to state the main Theorem of this section.

\begin{thm} \label{T:AC}
Let $B$ be an arbitrary weight $2$ block.  Let $\lambda$ be a partition in $B$, and let $\mu$ be a partition in $B'$.  Then
\begin{align*}
a_{\lambda\mu} &= \delta_{\lambda m(\mu)} \qquad \text{if $\mu$ is $e$-regular}.\\
a_{\lambda[0,1]_{B'}} &=
\begin{cases}
(-1)^{\del\lambda}, &\text{if } \varepsilon \lambda = 0; \\
(-1)^{\del\lambda+1}, &\text{if } \varepsilon \lambda = 1 \text{ and } \lambda \in \{ [e-1,b] \mid 0\leq b < e\}; \\
0, &\text{otherwise.}
\end{cases} \\
a_{\lambda[0,0]_{B'}} &=
\begin{cases}
(-1)^{\del\lambda}, &\text{if } \varepsilon \lambda = 0 \text{ and } \lambda \notin \{ [a,a] \mid 0 \leq a < e \}, \\
& \text{or }
\varepsilon \lambda = 1 \text{ and } \lambda \in \{ [e-1,b] \mid 0 \leq b \leq e \} \cup \{ [a,a+1] \mid 0 \leq a < e \}; \\
2(-1)^{\del\lambda}, &\text{if } \varepsilon \lambda = 1 \text{ and }\lambda \notin \{ [e-1,b] \mid 0 \leq b \leq e \} \cup \{ [a,a+1] \mid 0 \leq a < e \}; \\
0, &\text{otherwise.}
\end{cases} \\
a_{\lambda[j,0]_{B'}} &=
\begin{cases}
(-1)^{\del\lambda+j+1}, &\text{if } \lambda \in
\{ [a,e-j+\varepsilon \lambda] \mid e-j \leq a < e \} \cup \{ [e-j-\varepsilon \lambda,b] \mid \varepsilon \lambda \leq b < e-j \}; \\
0, &\text{otherwise}.
\end{cases}
\end{align*}
\end{thm}

Theorem \ref{T:AC} follows immediately from Lemma \ref{L:sequence} and the following two Propositions:

\begin{prop} \label{P:induction}
Suppose $B$ and $C$ are weight $2$ blocks forming a $[2:k]$-pair.  If Theorem \ref{T:AC} holds for $B$, then it holds for $C$.
\end{prop}

\begin{prop} \label{P:Rouq}
Theorem \ref{T:AC} holds for weight $2$ Rouquier blocks.
\end{prop}

\begin{proof}[Proof of Theorem \ref{T:AC} using Propositions \ref{P:induction} and \ref{P:Rouq}]
If $B$ is a Rouquier block, then the Theorem follows from Proposition \ref{P:Rouq}.  If $B$ is not Rouquier, then by Lemma \ref{L:sequence}, there exists a sequence $B_0,B_1,\dotsc,B_s$ of weight $2$ blocks such that $B_0 = B$, $B_s$ is Rouquier, and for each $1 \leq i \leq s$, there exists $k_i \in \mathbb{Z}^+$ such that $B_i$ and $B_{i-1}$ form a $[2:k_i]$-pair.  By induction, we may assume that the Theorem holds for $B_1$; hence it also holds for $B_0 = B$ by Proposition \ref{P:induction}.
\end{proof}

The following Lemma will be used in the proof of Proposition \ref{P:induction}.

\begin{lem} \label{L:list}
Suppose $B$ and $C$ form a $[2:1]$-pair, and let $\alpha = [a,b]$.  Then $c_{\alpha\lambda} = 1$ if and only if $\lambda = [a-1,b]$, or $[a,b+1]$, or $[a+1,b-1]$ (when $a \leq e-2$), or $[a+1,a+2]$ (when $a=b \leq e-2$).  \end{lem}

\begin{proof}
We prove this in four steps.

\textbf{Step 1. $c_{\alpha\lambda} = 1$ if and only if $d_{\wt{\alpha}\Phi(\lambda)}(v) = v$ or $d_{\Phi(\lambda)\wt{\alpha}}(v) = v$:}  This follows from Proposition \ref{P:equiv1} and Theorem \ref{T:ext}(1).

\textbf{Step 2. $d_{\Phi(\lambda)\wt{\alpha}}(v) = v$ if and only if $\lambda = {[a-1,b]}$:}  By Theorem \ref{T:21}(1,3,4), we have $d_{\Phi(\lambda)\wt{\alpha}}(v) = v$ if and only if $\lambda = \gamma$; furthermore $\gamma = [a-1,b]$ as $\alpha = [a,b]$.

For Steps 3 and 4, the set $S$ consists precisely of the partitions $[a,b+1]_B$, $[a+1,b-1]_B$ (when $a \leq e-2$) and $[a+1,a+2]_B$ (when $a=b\leq e-2$).

\textbf{Step 3. $d_{\wt{\alpha}\Phi(\lambda)}(v) = v$ if $\lambda \in S$:}
    \begin{description}
    \item[Case A. $\lambda = {[a, b+1]}$ with $b \leq 2$, or $\lambda = {[a+1,a+2]}$ with $a = b$ ($\leq e-2$)]  Note that $d_{\wt{\alpha}'\wt{\beta}'}(v) = v$ by Lemma \ref{L:commonwtB} and Theorems \ref{T:21}(3) and \ref{T:wt2vdecomp} (recall that $\gamma'$, $\beta'$ and $\alpha'$ are the exceptional partitions with respect to the $[2:1]$-pair $B'$ and $C'$, with $\gamma' \dom \beta' \dom \alpha'$).  Thus, when $\wt{\beta}$ is $e$-restricted (equivalently, when $b \leq e-2$), we have $d_{\wt{\alpha}m(\wt{\beta}')}(v) = v^2 d_{\wt{\alpha}'\wt{\beta}'}(v^{-1}) = v$ by Theorem \ref{T:Mull}.  Now, $\wt{\beta} = [a-1,b]_C$, and since an abacus display of $\wt{\beta}'$ can be obtained from that of $\wt{\beta}$ by rotating it through an angle of $\pi$ and reading the occupied positions as vacant and vacant position as occupied, we see that $\wt{\beta}' = [e-1-b,e-a]_{C'}$.  Thus,
        $$
        m(\wt{\beta}') = m([e-1-b,e-a]_{C'}) =
        \begin{cases}
        [a,b+1]_C, &\text{if } a > b; \\
        [a+1,a+2]_C, &\text{if } a = b
        \end{cases}
        $$
        by Lemma \ref{L:Mull}. This shows $d_{\wt{\alpha}\Phi(\lambda)}(v) = v$ when $a=b$ and $\Phi(\lambda) = [a+1,a+2]_C$ (which necessarily requires $a = b \leq e-2$), and when $\Phi(\lambda) = [a, b+1]_C$ with $b \leq e-2$.

    \item[Case B. $\lambda = {[a,b+1]}$ with $b = e-1$, or $\lambda = {[a+1,b-1]}$] We first describe the abacus displays of $\Phi(\lambda)$ with $N$ beads, where $N$ is chosen so that $\wt{\alpha}$ is obtained from its $e$-core by sliding the bottom bead on runner $(e-1)$ down two positions.  When $b=e-1$, no runner to the left of runner $(e-2)$ has more beads than runner $(e-1)$, and $[a,b+1]_C$ is obtained from its $e$-core by sliding the bottom bead of runner $(e-2)$ down two positions.  We illustrate this with an example.
$$
\begin{matrix}
\wt{\alpha} = [a,b]_C \\[6pt]
\begin{smallmatrix}
\bullet & \bullet & - & \bullet & - \\
- & - & - & \bullet & - \\
- & - & - & - & \bullet \\
- & - & - & - & -
\end{smallmatrix}
\end{matrix}
\qquad
\begin{matrix}
[a,b+1]_C \\[6pt]
\begin{smallmatrix}
\bullet & \bullet & - & \bullet & \bullet \\
- & - & - & - & - \\
- & - & - & - & - \\
- & - & - & \bullet & -
\end{smallmatrix}
\end{matrix}
$$
For $[a+1,b-1]_C$ (which necessarily requires $a \leq e-2$), there exists some runner having more beads than runner $(e-2)$.  Among these runners having more beads than runner $(e-2)$, let runner $r$ be the one having least number of beads, and if there are more than one such runner, let runner $r$ be the leftmost one.  Then $[a+1,b-1]_C$ is obtained from its $e$-core by sliding one bead each on runners $r$ and $(e-1)$ down one position.  Below is an example.
$$
\begin{matrix}
\wt{\alpha} = [a,b]_C \\[6pt]
\begin{smallmatrix}
\bullet & \bullet & \bullet & \bullet & - \\
- & \bullet & \bullet & \bullet & - \\
- & \bullet & \bullet & - & \bullet \\
- & - & - & - & -
\end{smallmatrix}
\end{matrix}
\qquad
\begin{matrix}
[a+1,b-1]_C \\[6pt]
\begin{smallmatrix}
\bullet & \bullet & \bullet & \bullet & - \\
- & \bullet & \bullet & \bullet & \bullet \\
- & - & \bullet & - & - \\
- & \bullet & - & - & -
\end{smallmatrix}
\end{matrix}
$$
From the descriptions of these partitions, it is easy to see that there does not exist any partition $\mu$ in $C$ satisfying $\Phi(\lambda) \dom \mu \dom \wt{\alpha}$, so that $J_{\wt{\alpha}\Phi(\lambda)} = 1$ (see Theorem \ref{T:Jantzen}), and hence $d_{\wt{\alpha}\Phi(\lambda)}(v) = v$ by Theorem \ref{T:R-H}.
\end{description}

\textbf{Step 4. $d_{\wt{\alpha}\Phi(\lambda)}(v) = v$ only if $\lambda \in S$:}  For this, we use Corollary \ref{C:cl} and the fact established in Step 3 that $d_{\wt{\alpha}\Phi(\lambda)}(v) = v$ if $\lambda \in S$.  Since $\del\wt{\alpha} = a-b+1$, we see that $\del\Phi(\lambda) = a-b$ or $a-b+2$ by Theorem \ref{T:wt2vdecomp}.  As $\del[a,b+1]_C = a-b$, and when $a=b$, $[a,b+1]_C$ ($= [a,a+1]_C$) and $[a+1,a+2]_C$ both have $\del$-value $0$, and are of different colour, we see that if $\del\Phi(\lambda) = a-b$, then $\Phi(\lambda) = [a,b+1]_C$ if $a \ne b$, while $\Phi(\lambda) \in \{ [a,b+1]_C, [a+1,a+2]_C \}$ if $a = b$ by Corollary \ref{C:cl}.  On the other hand, if $\del\Phi(\lambda) = a-b+2$, then $a \leq e-2$: this because when $a = e-1$, then no runner has more beads than runner $(e-2)$ and the partitions which dominate $\wt{\alpha}$ are obtained from their $e$-core by sliding the bottom bead of a runner which has the same number of beads as runner $(e-2)$ down two positions, and their $\del$-values can be checked to be bounded above by $a-b$.  Since $\del[a+1,b-1]_C = a-b+2$, we see that $\Phi(\lambda) = [a+1,b-1]_C$ by Corollary \ref{C:cl}, and the proof is complete.
\end{proof}

\begin{cor}[of proof] \label{C:epsilonvalues}
Suppose $B$ and $C$ form a $[2:1]$-pair, and let $\alpha = [a,b]$.  Then $\varepsilon[a-1,b] = 1$, $\varepsilon[a,b+1] = 1$, $\varepsilon[a+1,b-1] = 0$ (when $a \leq e-2$) and $\varepsilon[a+1,a+2] = 1$ (when $a =b \leq e-2$).
\end{cor}

\begin{proof}
We have seen the $\del$-values of these partitions, which thus enable us to compute their $\varepsilon$-values.
\end{proof}

\begin{proof}[Proof of Proposition \ref{P:induction}]
Let $\lambda$ be a partition in $B$ and let $\mu$ be a partition in $B'$.
Unless $k=1$, $\lambda = \alpha$ and $\mu \ne \gamma'$, we have $a_{\lambda\mu} = a_{\Phi(\lambda)\Phi(\mu)}$ by Proposition \ref{P:main}. As the labelling of weight $2$ partitions is invariant under the action of $\Phi$ by Lemma \ref{L:invar}, and $\varepsilon \lambda = \varepsilon \Phi(\lambda)$ when $\lambda \ne \alpha$ by Lemma \ref{L:ep}, we see that $a_{\Phi(\lambda)\Phi(\mu)}$ is as described in Theorem \ref{T:AC} if $a_{\lambda\mu}$ is.

When $k=1$, $\lambda = \alpha$ and $\mu \ne \gamma'$, it is routine to verify that $a_{\wt{\alpha}\Phi(\mu)}$ is as described in Theorem \ref{T:AC} when Theorem \ref{T:AC} holds for $B$ using Proposition \ref{P:alpha}, Lemma \ref{L:list} and Corollary \ref{C:epsilonvalues}.
\end{proof}

\begin{proof}[Proof of Proposition \ref{P:Rouq}]
This follows from the closed formulas obtained by Leclerc and Miyachi \cite[Corollary 10]{LM} for $e_{\lambda\mu}(v)$ when $\lambda$ and $\mu$ are canonical Rouquier partitions.  By Theorem \ref{T:vdecomp}(3,4) and Corollary \ref{C:decomp}, we are able to determine $d_{\lambda\mu}$ and $e_{\lambda\mu}$ when
$\lambda$ and $\mu$ are weight $2$ canonical Rouquier partitions, and use them to verify that Theorem \ref{T:AC} holds for the canonical Rouquier block.  Since the Rouquier blocks of a given weight form a single Scopes equivalence class, the Proposition follows from Proposition \ref{P:induction}.
\end{proof}

We conclude this paper with some equalities which we found in the course of studying the integers $a_{\lambda\mu}$.

\begin{prop}
Let $\lambda$ and $\mu$ be partitions.  Then
\begin{enumerate}
\item $\sum_{\nu} c_{\mu\nu} e_{\nu\lambda} = d_{\lambda\mu}$;
\item $\sum_{\nu} c_{\mu\nu} a_{\nu\lambda} = \sum_{\rho} d_{\rho\mu}d_{\rho'\lambda}$; in particular,
$$
\sum_{\nu} c_{\mu\nu} a_{\nu\lambda} =
\begin{cases}
c_{m(\lambda) \mu }, &\text{if $\lambda$ is $e$-regular}; \\
c_{\lambda m(\mu)}, &\text{if $\mu$ is $e$-regular}. \\
\end{cases}
$$
\end{enumerate}
\end{prop}

\begin{proof}
For (1), we have
$$
\sum_{\nu} c_{\mu\nu} e_{\nu\lambda}
= \sum_{\nu,\ \rho} d_{\rho\mu}d_{\rho\nu} e_{\nu\lambda}
= \sum_{\rho} d_{\rho\mu} \delta_{\rho\lambda}
= d_{\lambda\mu}.
$$
For (2), we have
$$
\sum_{\nu} c_{\mu\nu} a_{\nu\lambda} = \sum_{\nu,\ \rho} c_{\mu\nu} e_{\nu\rho}d_{\rho'\lambda} = \sum_{\rho} d_{\rho\mu}d_{\rho'\lambda}
$$
by (1).  If $\lambda$ is $e$-regular, then $d_{\rho'\lambda} = d_{\rho m(\lambda)}$ by Lemma \ref{L:basic}(2c), so that $
\sum_{\rho} d_{\rho\mu}d_{\rho'\lambda} = \sum_{\rho} d_{\rho\mu}d_{\rho m(\lambda)} = c_{\mu m(\lambda)} = c_{m(\lambda) \mu}$, while if $\mu$ is $e$-regular, then $d_{\rho\mu} = d_{\rho' m(\mu)}$, so that $\sum_{\rho} d_{\rho\mu}d_{\rho'\lambda} = \sum_{\rho} d_{\rho' m(\mu)}d_{\rho' \lambda} = c_{m(\mu) \lambda} = c_{\lambda m(\mu)}$.
\end{proof}

\end{document}